\documentclass[reqno]{amsart}
\usepackage{amsmath,amsthm,amssymb,amsfonts,mathrsfs,amscd}
\usepackage{mathtools}
\usepackage{varwidth}
\usepackage{hyperref}
\usepackage{multirow}
\usepackage{graphicx}
\usepackage{bm}
\usepackage{mathrsfs}
\usepackage{dsfont}
\usepackage{tikz}
\usepackage{tikz-3dplot}
\usepackage{subfig}
\usepackage{stmaryrd}
\usepackage{makecell}

\usepackage{float}
\usepackage[textfont=it]{caption}

\DeclareMathOperator\supp{supp}

\newtheorem{remark}{Remark}[section]
\newtheorem{assumption}{Assumption}[section]
\newtheorem{lemma}{Lemma}[section]
\newtheorem{theorem}{Theorem}[section]

\newtheoremstyle{claim}
  {}
  {}
  {\it}
  {}
  {\rm}
  {.}
  {.5em}
  {}
\theoremstyle{claim}

\makeatletter
  \@addtoreset{equation}{section}
  \newcommand\figcaption{\def\@captype{figure}\caption}
  \newcommand\tabcaption{\def\@captype{table}\caption}
\makeatletter

\begin{document}

\title{A novel coarse space applying to the weighted Schwarz method for Helmholtz equations}

\author{Qiya Hu}
\author{Ziyi Li}

 \thanks{1. LSEC, ICMSEC, Academy of Mathematics and Systems Science, Chinese Academy of Sciences, Beijing
 100190, China; 2. School of Mathematical Sciences, University of Chinese Academy of Sciences, Beijing 100049,
 China (hqy@lsec.cc.ac.cn, liziyi@lsec.cc.ac.cn). The work of the author was supported by the National Natural Science Foundation
of China grant G12071469.}

\maketitle

{\bf Abstract.} In this paper we are concerned with restricted additive Schwarz with local impedance transformation conditions
for a family of Helmholtz problems in two dimensions.
These problems are discretized by the finite element method with conforming nodal finite elements. We design and analyze a new adaptive coarse space
for this kind of restricted additive Schwarz method. This coarse space is spanned by some eigenvalue functions of local generalized eigenvalue problems,
which are defined by weighted positive semi-definite bilinear forms on subspaces consisting of local discrete Helmholtz-harmonic functions from impedance boundary data.
We proved that a two-level hybrid Schwarz preconditioner with the proposed coarse space possesses uniformly convergence independent of the mesh size, the subdomain size and the wave
numbers under suitable assumptions. We also introduce an economic coarse space to avoid solving generalized eigenvalue problems. Numerical experiments confirm the theoretical results.



\pagestyle{myheadings}
\thispagestyle{plain}
\markboth{}{Qiya Hu and Ziyi Li}

\section{Introduction}

Let $\Omega$ be a bounded, connected and Lipschitz domain in $\mathbb{R}^2$. Consider the following Helmholtz equation with absorption:
\begin{equation} \label{eq:shiftHelmholtz}
    \left\lbrace
    \begin{aligned}
         & -\Delta u - (\kappa^2 + \mathrm{i}\epsilon) u = f
         & \text{in}\ \Omega,                                               \\
         & \frac{\partial u}{\partial {\bf n}} - \mathrm{i} \kappa u = g
         & \text{on}\ \partial\Omega,
    \end{aligned}
    \right.
\end{equation}
where ${\bf n}$ denotes the unit outward normal on the boundary $\partial\Omega$, $\kappa>0$ is the wave number and $\epsilon > 0$ is a absorption parameter. The pure Helmholtz equation
with $\epsilon=0$ can be approximated or preconditioned by the equation (1.1) with some $\epsilon>0$, see \cite{gander2015applying} for the details.
Then we need only to consider (1.1) with $\epsilon>0$.

The Helmholtz equation is the basic mathematical model in the numerical simulation for sound propagation. A finite element method for the discretization of
(1.1) often has the so-called ``wave number pollution", so the mesh size $h$ must satisfy $h\kappa^{1+\gamma}=O(1)$ for some positive number $\gamma$
so that the approximate solution could achieve a high accuracy when the wave number $\kappa$ increases, see \cite{DuWu2015} and \cite{melenk2001wavenumber}. This means that the mesh size $h$ should be very small and so
the resulting algebraic system is huge when the wave number $\kappa$ is large. Moreover, this algebraic system is highly indefinite and severely ill-conditioned. All these indicate that solving the
linear system is very expensive and difficult.

It is well known that domain decomposition methods are powerful parallel algorithms for solving the systems generated by finite element discretization of elliptic-type
partial differential equations (see  \cite{toselli2005domain} and the references therein). However, the standard domain decomposition methods are inefficient for Helmholtz equations with large
wave numbers unless the size $d$ of the coarse mesh is chosen such that $d\kappa$ is sufficiently small (which is denoted by $d\kappa=O(1)$ for simplicity), see \cite{cai1992domain} and
\cite{graham2017domain}. It is clear that the
restriction on coarse mesh sizes is fatal in applications. The restriction $d\kappa=O(1)$ can be relaxed to some extent for particular discretization methods and particular domain decomposition methods,
see \cite{HuZ2016,HuS2020}.

The restricted additive Schwarz with local impedance conditions for Helmholtz equations was proposed in \cite{kimn2007restricted} and \cite{stcyr2007optimized}.
This method falls into the overlapping domain decomposition method,  but it has two differences from the standard overlapping domain decomposition method:
(i) local problems are defined by impedance boundary conditions instead of Dirichlet boundary conditions;
(ii) extra weight operators associated with partition of unity functions are placed on the left side of the one-level additive preconditioner
, for which this method can be viewed in form (but not spirit) as an extension
of the restricted additive Schwarz method proposed in \cite{cai1999restricted}. For convenience, we abbreviate this method as WASI, which seems more precise than the other abbreviations
(for example, ORAS, WRAS).
In recent years, the WASI method has attracted many attentions from researchers since it shows higher efficiency in solving Helmholtz equations and may essentially weaken the restriction $d\kappa=O(1)$.
Convergence results of the one-level WASI method for two-dimensional Helmholtz problems were established at the continuous level and for discrete problems
in \cite{gong2022convergence} and \cite{gong2023convergence}, respectively. In these works power contractility of the error propagation operators
was built by carefully analyzing the so called ``impedance-to-impedance operators'' for the domain decompositions in strips. A one-level additive Schwarz method with local impedance conditions
but using double-side weight operators for Helmholtz equations with some absorption was considered and analyzed in \cite{graham2020domain}, which established a convergence result for more general domain decomposition (not in strips).

A coarse space generated from plane wave functions was also introduced in \cite{kimn2007restricted}.
The adaptive coarse spaces spanned by some eigenfunctions of local generalized eigenvalue problems are popular. For example, adaptive coarse spaces for overlapping
domain decomposition method solving positive definite problems were constructed in \cite{heinlein2019adaptive}, \cite{lu2022two} and \cite{spillane2014abstract}. The key technique in the construction of
adaptive coarse spaces is the design of local generalized eigenvalue problems, which depend on the considered models and domain decomposition methods.
Two adaptive coarse spaces for WASI method were proposed in \cite{bootland2021comparison} and \cite{conen2014coarse}, in which local generalized eigenvalue problems are defined by Helmholtz sesquilinear forms. But the generalized eigenvalue problems considered in \cite{bootland2021comparison} and \cite{conen2014coarse} are defined on different subspace: the natural restriction spaces are considered in
\cite{bootland2021comparison} and the spaces consisting of local Helmholtz-harmonic functions (with Dirichlet boundary data) are considered in \cite{conen2014coarse}.
To our knowledge, there is no work to give rigorous convergence analysis for two-level WASI method with a coarse solver.

In the current paper, we construct and analyze a new adaptive coarse space for the WASI method solving Helmholtz equations. For this coarse space, every local generalized eigenvalue problem
is defined by a weighted positive semi-definite bilinear form posed on the subspace consisting of local Helmholtz-harmonic functions from impedance boundary data.
We prove in theory that a two-level hybrid WASI method with the proposed coarse space possesses uniform convergence independent of the mesh size $h$, the subdomain size $d$ and the wave number $\kappa$, provided that rational assumptions are satisfied. We also introduce an economic version of the proposed coarse space so that the solution of generalized eigenvalue problems can be avoided.
Numerical experiments confirm the efficiency of the proposed coarse spaces.

The paper is organized as follows: In Section 2, we introduce variational problems and some basic auxiliary results.
In Section 3, we define a coarse space and construct a two-level overlapping domain decomposition preconditioner with the resulting coarse solver.
In Section 4, we establish several auxiliary results for the analysis of convergence results.
In Section 5 and Section 6, we give convergence analysis of the preconditioner for a slightly large absorption parameter and a small absorption parameter, respectively.
In Section 7, we compare several related coarse spaces, including the economic coarse space. Finally, we report some numerical results to confirm the efficiency of
the proposed coarse space.

\section{Preliminary}

\subsection{Variational problem and its discretization}
Let $(\cdot,\cdot)$ and $\langle\cdot,\cdot\rangle$ denote the $L^2(\Omega)$ inner product and the $L^2(\partial\Omega)$ inner product, respectively. We use $H^1(\Omega)$ to denote the standard Sobolev space consisting of
weakly derivable functions on $\Omega$. For a complex function $v\in H^1(\Omega)$, let $\bar{v}$ denote its conjugate function. For a given parameter $\epsilon>0$, define the sesquilinear form $a_{\epsilon}(.,.)$ by
\begin{equation*}
    a_{\epsilon}(v,w) \coloneqq (\nabla v,\nabla \bar{w}) - (\kappa^2+\mathrm{i}\epsilon)(v,\bar{w}) -\mathrm{i}\kappa\left\langle v,\bar{w}\right\rangle, \quad v,w\in H^1(\Omega).
\end{equation*}

Assume that $f\in L^2(\Omega)$ and $g\in L^2(\partial\Omega)$. The standard variational problem of \eqref{eq:shiftHelmholtz} is as follows: find $u\in H^1(\Omega)$, such that
\begin{equation}\label{eq:shiftHelmholtzVF}
    a_{\epsilon}(u,v) = F(v), \quad \forall v\in H^1(\Omega),
\end{equation}
where
\begin{equation*}
    F(v) \coloneqq (f,\bar{v})+\int_{\partial\Omega}g\bar{v}ds.
\end{equation*}

Let $\mathcal{T}_h$ be a mesh posed on $\Omega$ with mesh diameter $h$. As usual, we assume that $\mathcal{T}_h$ is quasi-uniform and shape-regular. Let $V_h(\Omega)$ be the nodal finite element space with fixed order $p$
built on $\mathcal{T}_h$. As usual, Galerkin method could be used to find the approximate solution of $u$ in $V_h(\Omega)$.

The discrete variational problem of (\ref{eq:shiftHelmholtzVF}) is: find $u_h\in V_h(\Omega)$, such that
\begin{equation}\label{eq:discreteVF}
    a_{\epsilon}(u_h,v_h) = F(v_h),\quad \forall v_h\in V_h(\Omega).
\end{equation}

It is well known that the convergence of $u_h$ has the so called ``wave number pollution", which means that the mesh size $h$ should satisfy $h\ll\kappa^{-1}$, see \cite{melenk2001wavenumber}.
Because of this, we give an assumption on $h$.
\begin{assumption}\label{assump:meshsize} The mesh size $h$ satisfies  $h\sim \kappa^{-(1+\gamma)}$ for some $0<\gamma\leq 1$.

\end{assumption}
In general the value of $\gamma$ depends on the order $p$: the larger $p$ is, the smaller $\gamma$ is, see  \cite{DuWu2015}.

For convenience, define the discrete operator $A_\epsilon$ as
\begin{equation*}
    (A_\epsilon v_h, w_h) = a_{\epsilon}(u_h,v_h),\quad\forall v_h,w_h\in V_h(\Omega).
\end{equation*}
Then the equation (\ref{eq:discreteVF}) can be written as the operator form
\begin{equation}
    A_\epsilon u_h=f_h,\quad u_h, f_h\in V_h(\Omega).\label{new2.1}
\end{equation}

\subsection{Notations}

We introduce the standard $\kappa$-weighted inner product and norm on $H^1(\Omega)$:
\begin{equation*}
    (v,w)_{1,\kappa} = (\nabla v,\nabla \bar{w}) + \kappa^2(v,\bar{w})\quad\text{and}\quad \Vert v\Vert_{1,\kappa} = (v,v)_{1,\kappa}^{\frac{1}{2}}.
\end{equation*}
Let $\Vert\cdot\Vert_{0,\Omega}$ and $\Vert\cdot\Vert_{0,\partial\Omega}$ denote the norms on $L^2(\Omega)$ and $L^2(\partial\Omega)$ respectively.

For a subdomain $G$ of $\Omega$, where $G$ is just the union of some elements in ${\mathcal T}_h$, let $V_h(G)$ denote the finite element space restricted on $G$:
$$ V_h(G):=\{v_h|_{G}: v_h\in V_h(\Omega)\}. $$
The trace space of $V_h(G)$ is denoted by $V_h(\partial G)$:
$$ V_h(\partial G):=\{v_h|_{\partial G}: v_h\in V_h(G)\}. $$
Let $V^0_{h}(G) \subseteq V_{h}(G)$ be the space consisting of functions vanishing on the inner boundary $\partial G\setminus\partial\Omega$, namely,
$$ V^0_{h}(\Omega_l)\coloneqq\{v_h\in V_h(G):~v_h=0\text{ on }\partial G\setminus\partial\Omega\}. $$
Define a local inner product on $G$ as
$$ (v, w)_{1,\kappa,G}=\int_G\nabla v\cdot\nabla \bar{w} dx+\kappa^2\int_{G}v\bar{w}dx,\quad v,w\in H^1(G). $$
Let $\Vert\cdot\Vert_{1,\kappa,G}$ denotes the norm induced from the inner product $(\cdot, \cdot)_{1,\kappa,G}$.

Throughout this paper, we will use the notation $a\lesssim b$ to represent that there exists a $C>0$ such that $a\leq C b$, and the notation $a\sim b$ to represent that $a\lesssim b$ and $b\lesssim a$. For a domain $D\in\mathbb{R}^n$, by saying $D$ has characteristic length scale $L$ means that its diameter $\sim L$, the surfaces area $\sim L^{n-1}$ and its volume $\sim L^{n}$.

Let $\mathcal{N}(G)$ be the set of all finite element nodes on $\overline{G}$. For any $v_h\in V_h(G)$, the discrete $L^2$ norm of $v_h$ is defined as
\begin{equation}\label{eq:discreteL2Norm}
    \Vert v_h\Vert_{0,h,G}^2 = h^2\sum_{x\in \mathcal{N}(G)} |v_h(x)|^2.
\end{equation}
There holds
\begin{equation}\label{eq:dL2NormSim}
    \Vert v_h\Vert_{0,h,G} \sim \Vert v_h\Vert_{0,G}.
\end{equation}

\subsection{Some basic auxiliary results}

The following multiplicative trace inequality provides a $L$-explicit estimate of traces:
\begin{lemma}{\rm(Multiplicative trace inequality)}
    Let $G$ be a Lipschitz domain with characteristic length scale $L$, then for every $v\in H^1(G)$ there holds
    \begin{equation}\label{eq:multiTrace}
        \Vert v\Vert^2_{0,\partial G} \lesssim L^{-1}\Vert v\Vert_{0,G}^2 + \Vert \nabla v\Vert_{0,G}\Vert u\Vert_{0,G}
    \end{equation}
\end{lemma}
Thanks to the absorption parameter $\epsilon$, sesquilinear form $a_\epsilon(\cdot,\cdot)$ is coercive:
\begin{lemma}{\rm(Continuity and coercivity,\cite[Lemma 3.1]{gander2015applying})}\label{lemma:contCoer}
    \begin{itemize}
        \item[(1)] Let $\Omega$ have characteristic length scale $L$ and $0\leq\epsilon\leq\kappa^2$. Then the sesquilinear form $a_\epsilon(\cdot,\cdot)$ is continuous:
              \begin{equation}\label{eq:continuous}
                  \vert a_\epsilon(v,w)\vert \lesssim (1+(\kappa L)^{-1}) \Vert v\Vert_{1,\kappa} \Vert w\Vert_{1,\kappa},\quad \forall v,w\in H^1(\Omega).
              \end{equation}
        \item[(2)] $a_{\epsilon}(\cdot,\cdot)$ is also coercive:
              \begin{equation}\label{eq:coercive}
                  \vert a_\epsilon(v,v)\vert \gtrsim \frac{\epsilon}{\kappa^2} \Vert v\Vert_{1,\kappa}^2,\quad \forall v\in H^1(\Omega)
              \end{equation}
    \end{itemize}
\end{lemma}

\begin{lemma}\label{lemma:stability}
    Let $\Omega$ have characteristic length scale $L$ and be star-shaped with respect to a disk of radius $\sim L$. Let $u$ be the solution to \eqref{eq:shiftHelmholtz} with $f\in L^2(\Omega)$ and $g=0$. Then, $u$ satisfies
    \begin{equation}\label{eq:stability}
        \Vert u\Vert_{1,\kappa}\lesssim L\Vert f\Vert_{\Omega}
    \end{equation}
\end{lemma}
\begin{proof}
    This result is a natural corollary of \cite[Theorem 2.7]{graham2020domain} and \cite[Corollary 2.8]{gander2015applying}. By \cite[Theorem 2.7]{graham2020domain}, \eqref{eq:stability} holds when$\frac{\epsilon}{\kappa}\leq \frac{C_2}{L}$ where $C_2$ is the constant defined therein. And when $\frac{\epsilon}{\kappa}> \frac{C_2}{L}$, inserting it into \cite[Corollary 2.8]{gander2015applying}, we obtain
    \begin{equation*}
        \Vert u\Vert_{1,\kappa}\lesssim \frac{L}{C_2}\Vert f\Vert_{\Omega}.
    \end{equation*}
\end{proof}

The well-posedness of the variational problem (\ref{eq:discreteVF}) is guaranteed by the following lemma. Combining \eqref{eq:coercive} and \eqref{eq:stability}, the sesquilinear form $a_{\epsilon}(\cdot,\cdot)$ satisfies the discrete inf-sup condition:
\begin{lemma}{\rm (\cite[Theorem 2.8(ii)]{graham2020domain})}\label{lemma:infsup}
    Let the assumption made in Lemma \ref{lemma:stability} be satisfied. Then there exists a mesh threshold function $\overline{h}(\kappa,p)$ such that when mesh size $h < \overline{h}(\kappa,p)$, we have the estimate
   \begin{equation}\label{eq:inf-sup}
    \Vert v_h\Vert_{1,\kappa}  \lesssim  \min\left\lbrace 1+\kappa L,\frac{\kappa^2}{\epsilon}\right\rbrace\sup_{0\neq w_h\in V_h(\Omega)}\frac{\vert a_\epsilon(v_h,w_h)\vert}{\Vert w_h\Vert_{1,\kappa}},\quad\forall v_h\in V_h(\Omega).
    \end{equation}
 If the star-shaped assumption in Lemma \ref{lemma:stability} is not satisfied, then we only have
    \begin{equation}\label{eq:inf-sup1}
      \Vert v_h\Vert_{1,\kappa}  \lesssim  \frac{\kappa^2}{\epsilon}\sup_{0\neq w_h\in V_h(\Omega)}\frac{\vert a_\epsilon(v_h,w_h)\vert}{\Vert w_h\Vert_{1,\kappa}},\quad\forall v_h\in V_h(\Omega).
    \end{equation}
\end{lemma}
\begin{remark} When $a_{\epsilon}(\cdot,\cdot)$ in Lemma \ref{lemma:contCoer} and Lemma \ref{lemma:infsup} are replaced by its adjoint sesquilinear form $\bar{a}_{\epsilon}(\cdot,\cdot)$:
$$  \bar{a}_{\epsilon}(v,w) \coloneqq (\nabla v,\nabla \bar{w})-(\kappa^2-\mathrm{i}\epsilon)(v,\bar{w}) +\mathrm{i}\kappa\left\langle v,\bar{w}\right\rangle, \quad v,w\in H^1(\Omega),  $$
the results still hold. Besides, the domain $\Omega$ in Lemma \ref{lemma:contCoer} and Lemma \ref{lemma:infsup} can be replaced by its subdomain $G$.
\end{remark}

\section{A domain decomposition preconditioner}
In this section, we are devoted to constructing a two-level domain decomposition preconditioner $B_{\epsilon}^{-1}$ for $A_\epsilon$. Specifically, we are going to construct
a coarse space to improve the robustness of the one level precondition proposed in \cite{kimn2007restricted} and \cite{stcyr2007optimized}.

\subsection{One level settings}\label{subsec:1level}

Let $\overline{\Omega}=\cup_{l=1}^N \overline{\Omega'_l}$ be a non-overlapping domain decomposition of $\Omega$. Each $\Omega'_l$, with diameter $d_l\sim d$, consists of a union of elements of the mesh $\mathcal{T}_h$. We can extend each $\Omega'_l$ by several layers of elements to a larger region $\Omega_l$. Then $\{\Omega_l\}_{l=1}^N$ form an overlapping domain decomposition of $\Omega$ and we denote the overlap by $\delta$. For a subdomain $\Omega_l$, define
    \begin{equation*}
        \Lambda(l)\coloneqq \{j:~ \Omega_j\cap\Omega_l \neq \emptyset\}.
    \end{equation*}
Namely, $\Lambda(l)$ consists of the indices of the subdomains that intersect with $\Omega_l$.

We make the following natural assumptions:
\begin{assumption}\label{assump:finiteOverlap}
(i) Every subdomain $\Omega_l$ is star-shaped with respect to a disc.

(ii)
The number of indices in $\Lambda(l)$ is uniformly bounded with respect to $l$, i.e.,
    \begin{equation*}
        \#\Lambda(l)\lesssim 1,\quad \forall 1\leq l\leq N.
    \end{equation*}
\end{assumption}

For each $\Omega_l$, we use $R_l: V_h(\Omega)\rightarrow V_h(\Omega_l)$ to denote the natural restriction operator: $R_l v_h = v_h\vert_{\Omega_l}$. Then the finite element space on $\Omega_l$ can be written as
\begin{equation*}
    V_h(\Omega_l)\coloneqq \{ R_l v_h:~ v_h\in V_h(\Omega)\}.
\end{equation*}
We can also define the prolongation operator $E_l:V_h(\Omega_l)\rightarrow V_h(\Omega)$ as follows: for $v_h\in V_h(\Omega_l)$, $E_l v_h\in V_h(\Omega)$ is defined by its nodal values satisfying
\begin{equation*}
    (E_l v_h)(x) =
    \left\lbrace
    \begin{aligned}
         & v_h(x),\quad &  & \text{if }x\in\overline{\Omega}_l, \\
         & \quad 0,\quad      &  & \text{otherwise},
    \end{aligned}
    \right.
\end{equation*}
where $x$ is any finite element node on $\overline{\Omega}$. Besides, we use $E_l^*$ to denote the adjoint operator of $E_l$ with respect to $(\cdot,\cdot)$.

As in \cite{graham2020domain}, local variational problem on each $\Omega_l$ is solved. For general $F_l\in (H^1(\Omega_l))'$, the local variational problem is: find $u_l\in H^1(\Omega_l)$ such that
\begin{equation}\label{eq:localVF}
    a_{\epsilon,l}(u_l,v_l) = F_l(v_l),\quad \forall v_l\in H^1(\Omega_l),
\end{equation}
where the local sesquilinear form $a_{\epsilon,l}(\cdot,\cdot)$ with impedance boundary condition is defined by
\begin{equation}
    a_{\epsilon,l}(u_l,v_l) \coloneqq (\nabla u_l,\nabla \bar{v}_l)_{\Omega_l} - (\kappa^2+\mathrm{i}\epsilon)(u_l,\bar{v}_l)_{\Omega_l} -\mathrm{i}\kappa\left\langle u_l,\bar{v}_l\right\rangle_{\partial\Omega_l}.
\end{equation}
Then the local discrete variational problem is as follows: find $u_{h,l}\in V_h(\Omega_l)$ such that
\begin{equation}\label{eq:localDiscreteVF}
    a_{\epsilon,l}(u_{h,l},v_{h,l}) = F_l(v_{h,l}),\quad \forall v_{h,l}\in V_h(\Omega_l).
\end{equation}
The corresponding operator $A_{\epsilon,l}$ is defined as
\begin{equation*}
    (A_{\epsilon,l} u_{h,l},v_{h,l}) = a_{\epsilon,l}(u_{h,l},v_{h,l}),\quad\forall u_{h,l},v_{h,l}\in V_h(\Omega_l).
\end{equation*}

Note that $V_h(\Omega_l)$ is not a subspace of $V_h(\Omega)$. A partition of unity $\{\chi_l\}_{l=1}^N$ is used to connect the local problems. $\chi_l:\overline{\Omega}\rightarrow\mathbb{R}$ satisfies
\begin{equation*}
    \supp \chi_l \subseteq \overline{\Omega}_l,\quad 0\leq\chi_l\leq 1 \quad\text{and}\quad \sum_{l=1}^{N}\chi_l = 1.
\end{equation*}
Furthermore, we assume that
\begin{equation*}
    \vert\nabla\chi_l\vert\lesssim\delta^{-1}.
\end{equation*}

Let $\pi_h: C^0(\Omega)\rightarrow V_h(\Omega)$ denote the standard nodal interpolation. Define $D_l:V_h(\Omega)\rightarrow V_h(\Omega)$ as
\begin{equation*}
    D_l(v_h) = \pi_h(\chi_l v_h),\quad \forall v_h\in V_h(\Omega).
\end{equation*}
The one-level weighted additive preconditioner introduced in \cite{kimn2007restricted} and \cite{stcyr2007optimized} is written as
\begin{equation}
    B_{\epsilon,WASI}^{-1} = \sum_{l=1}^{N} D_l E_l A_{\epsilon,l}^{-1}E_l^*. \label{new4.0}
\end{equation}

\subsection{A coarse space}\label{subsec:coarseSpace}

In this subsection we design a novel coarse space and combine it with WASI in a hybrid style. The basic idea of the construction of the coarse space can be described as follows:
At first, a generalized eigenvalue problem posed on the local discrete Helmholtz-harmonic space is solved for every subdomain. Next, some eigenfunctions for each subdomain are selected
and extended as functions in $V_h(\Omega)$ by applying $D_lE_l$ and making up the coarse space.

For $\lambda_h\in V_h(\partial\Omega_l)$, we use ${\mathcal E}_{\epsilon, l}(\lambda_h)\in V_h(\Omega_l)$ to denote the Helmholtz-harmonic extension of the
impedance boundary data $\lambda_h$, in the sense that
$$ a_{\epsilon,l}({\mathcal E}_{\epsilon, l}(\lambda_h),w_h) =\langle\lambda_h,w_h\rangle_{\partial\Omega_l},\quad \forall w_h\in V_{h}(\Omega_l). $$
The function ${\mathcal E}_{\epsilon, l}(\lambda_h)$ is always well-defined. The local discrete Helmholtz-harmonic space with impedance boundary data is defined as
\begin{equation}\label{eq:discreteHarmonic-new}
    V_h^{\partial}(\Omega_l)\coloneqq \{ v_h\in V_h(\Omega_l):~v_h={\mathcal E}_{\epsilon, l}(\lambda_h)~~\mbox{for~some}~~\lambda_h\in V_h(\partial\Omega_l)\}.
\end{equation}
This space is always well-defined.
Define the weighted Helmholtz-harmonic space
\begin{equation*}
    V^{\partial}_{h}(\Omega):=\{v_h=\sum_{l=1}^N D_l E_l v_l:~v_l\in V_{h}^{\partial}(\Omega_l)\}.
\end{equation*}

A generalized eigenvalue problem posed on $V_h^{\partial}(\Omega_l)$ can be introduced: find $\xi\in V_h^{\partial}(\Omega_l)$ such that
\begin{equation}\label{eq:geEigenvalue}
    (D_l E_l \xi,D_l E_l \theta)_{1,\kappa,\Omega_l} = \lambda (\xi,\theta)_{1,\kappa,\Omega_l},\quad \forall\theta\in V_h^{\partial}(\Omega_l).
\end{equation}
We need to emphasize that the inner-product $(D_l E_l\cdot, D_l E_l \cdot)_{1,\kappa,\Omega_l}$ is different from that used in \cite{bootland2021comparison} and \cite{conen2014coarse}
(see Section 6 for the details).

Denote the dimension of $V_h^\partial(\Omega_l)$ by $m_l$. Let $\{(\lambda_l^i,\xi_l^i)\}_{i=1}^{m_l}$ be the eigenpairs of \eqref{eq:geEigenvalue} satisfying
\begin{equation*}
    0\leq\lambda_l^1\leq\lambda_l^2\leq\cdots\leq\lambda_l^{m_l} \text{ and } (\xi_l^i,\xi_l^j)_{1,\kappa,\Omega_l}=
    \left\lbrace
    \begin{aligned}
        1,\, i=j \\
        0,\, i\neq j
    \end{aligned}.
    \right.
\end{equation*}
Then every $v\in V^\partial_h(\Omega_l)$ has a decomposition
\begin{equation}\label{eq:eigenDecomp}
    v = \sum_{i=1}^{m_l} (v,\xi_l^i)\xi_l^i.
\end{equation}
For a given tolerance $\rho\in (0, 1)$, let $m_l^\rho$ be such that $\lambda_l^{m_l^\rho-1}<\rho^2\leq\lambda_l^{m_l^\rho}$. Finally we can define the local coarse space $V^{\rho}_{h,0}(\Omega_l)$ as
\begin{equation*}
    V^{\rho}_{h,0}(\Omega_l) \coloneqq \text{span}\{\xi_l^i:~ i\geq m_l^\rho\}
\end{equation*}
and the projection $\Pi_l^\rho$ as
\begin{equation*}
    \Pi_l^\rho v \coloneqq \sum_{i=m_l^\rho}^{m_l}(v,\xi_l^i)\xi_l^i,\quad \forall v\in V^\partial_h(\Omega_l)\text{ satisfying\eqref{eq:eigenDecomp}}.
\end{equation*}

\begin{lemma}\label{lemma:eigenEstimate}
    The projection $\Pi_l^\rho$ is actually the $(\cdot,\cdot)_{1,\kappa,\Omega_l}$-orthogonal projection from $V^\partial_h(\Omega_l)$ to $V_{h,0}^{\rho}(\Omega_l)$. In addition, we have the stability estimate
    \begin{equation}\label{eq:eigenEstimate}
        \Vert D_l E_l(I-\Pi_l^\rho)v_h\Vert_{1,\kappa,\Omega_l}\leq\rho\Vert v_h\Vert_{1,\kappa,\Omega_l},\quad\forall v_h\in V^\partial_h(\Omega_l).
    \end{equation}
\end{lemma}
\begin{proof}
    It is obvious that $(D_lE_l\cdot,D_lE_l\cdot)_{1,\kappa,\Omega_l}$ is positive semi-definite and $(\cdot,\cdot)_{1,\kappa,\Omega_l}$ is positive definite on $V_h^\partial(\Omega_l)$. The $(\cdot,\cdot)_{1,\kappa,\Omega_l}$-orthogonality of $\Pi_l^\rho$ follows from $(\cdot,\cdot)_{1,\kappa,\Omega_l}$-orthogonality of $\xi_l^i$ and the definition of $\Pi_l^\rho$. The estimate \eqref{eq:eigenEstimate} is just the property of the standard generalized Hermitian eigenvalue problem, see Lemma 2.11 in \cite{spillane2014abstract}.
\end{proof}
The coarse space is defined by
\begin{equation*}
    V^{\rho}_{h,0}(\Omega):=\{v_h=\sum_{l=1}^N D_l E_l v_l:~v_l\in V_{h,0}^{\rho}(\Omega_l)\}.
\end{equation*}
Since $ V^{\rho}_{h,0}(\Omega)$ is a subspace of $V^{\partial}_h(\Omega)$, we can define the restriction of $A_\epsilon$ on $V^{\rho}_{h,0}(\Omega)$ as $A^{\rho}_{0}:V^{\rho}_{h,0}(\Omega)\rightarrow V^{\rho}_{h,0}(\Omega)$ satisfying
\begin{equation*}
    (A^{\rho}_{0} u_0,v_0) = a_{\epsilon}(u_0,v_0),\quad\forall u_0,v_0\in V^{\rho}_{h,0}(\Omega)
\end{equation*}
and introduce an identical lifting operator $E_0:V^{\rho}_{h,0}(\Omega)\rightarrow V_h(\Omega)$.

Finally, a two-level hybrid preconditioner is defined as
\begin{equation}\label{eq:2levelPre}
    B_{\epsilon}^{-1} = (I-E_0 (A^{\rho}_{0})^{-1} E_0^* A_{\epsilon})B_{\epsilon, WASI}^{-1} + E_0 (A^{\rho}_{0})^{-1} E_0^*,
\end{equation}
where $I$ denotes the identical operator and $E_0^*$ denotes the adjoint operator of $E_0$ with respect to $(\cdot,\cdot)$. Then, from (\ref{new4.0}) we have
\begin{equation}
    B_{\epsilon}^{-1}A_{\epsilon}= (I-E_0 (A^{\rho}_{0})^{-1} E_0^* A_{\epsilon})\sum_{l=1}^{N} D_l E_l A_{\epsilon,l}^{-1}E_l^*A_{\epsilon}+ E_0 (A^{\rho}_{0})^{-1} E_0^* A_{\epsilon}.\label{new4.2}
\end{equation}



\subsection{Projection operators}\label{subsec:projOps}
In order to analyze the preconditioner \eqref{eq:2levelPre}, we need to simplify the form of $B_{\epsilon}^{-1}A_{\epsilon}$ by introducing several projection operators.

First of all, for $l=1,...,N$, the local projection operators $P_{\epsilon,l}: V_h(\Omega)\rightarrow V_h(\Omega_l)$ are defined as follows: for each $v_h\in V_h(\Omega)$, $P_{\epsilon,l} v_h\in V_h(\Omega_l)$ is defined to be the solution of the local problem
\begin{equation}\label{eq:localProj}
    a_{\epsilon,l}(P_{\epsilon,l}v_h,w_{h,l}) = a_{\epsilon}(v_h,E_l w_{h,l}),\quad \forall w_{h,l}\in V_h(\Omega_l).
\end{equation}
It is easy to see that $P_{\epsilon,l}=A_{\epsilon,l}^{-1} E_l^* A_{\epsilon}$. The coarse projection operator $P^{\rho}_0: V_h(\Omega)\rightarrow V^{\rho}_{h,0}(\Omega)$ could be defined by the same way. For each $v_h\in V_h(\Omega)$, define $P^{\rho}_0 v_h\in V^{\rho}_{h,0}(\Omega)$
as the solution of the problem
\begin{equation}\label{eq:coarseProj}
    a_{\epsilon}(P^{\rho}_0v_h,w_{h,0}) = a_{\epsilon}(v_h, w_{h,0}),\quad \forall w_{h,0}\in V^{\rho}_{h,0}(\Omega).
\end{equation}
Then we have $P^{\rho}_0=E_0 (A^{\rho}_{0})^{-1} E_0^* A_{\epsilon}$.

The well-posedness of \eqref{eq:localProj} and \eqref{eq:coarseProj} is guaranteed by \eqref{eq:coercive}. Finally, the global projection operator $P^{\rho}_{\epsilon}: V_h(\Omega)\rightarrow V_h(\Omega)$ is defined by
\begin{equation}\label{eq:globalProj}
    P^{\rho}_{\epsilon}= (I - P^{\rho}_{0})\sum_{l=1}^{N} D_l E_l P_{\epsilon,l} + P^{\rho}_{0}.
\end{equation}
It is easy to check that $P^{\rho}_{\epsilon} = B_\epsilon^{-1}A_\epsilon$.

In the forthcoming three sections,
we will estimate a lower bound of $(v_h,P^{\rho}_{\epsilon}v_h)$ and a upper bound of $\Vert P^{\rho}_{\epsilon}v_h\Vert_{1,\kappa}$ for any $v_h\in V_h(\Omega)$. These bounds
can reveal the convergence rate of the preconditioned GMRES method with the preconditioner $B_{\epsilon}^{-1}$.
The analysis is based on an important observation: for any $v_h\in V_h(\Omega)$, the global projection $P^{\rho}_{\epsilon}v_h$ can be expressed as
\begin{equation}\label{eq:observ}
    P^{\rho}_{\epsilon}v_h=(I - P^{\rho}_{0})v_h^{\partial}+v_h\quad\mbox{with}~~v_h^{\partial}\in V_h^{\partial}(\Omega),
\end{equation}
which will play a crucial rule in the analysis. Thanks to this equality, one needs only to investigate the approximation of the operator $P^{\rho}_{0}$ on the subspace $V_h^{\partial}(\Omega)$.

\section{Some auxiliary results}\label{sec:auxiliary}
In this section, we establish several technical lemmas, which will be used to prove the main results given later.

At first we give a simple result, which shows the equivalence of two Helmholtz-harmonic spaces in theory (not in numerical aspect).\\
{\bf Proposition 4.1}. Let $\tilde{V}_h^{\partial}(\Omega_l)$ denote the local Helmholtz-harmonic space with Dirichlet boundary data, which is defined as (refer to \cite{conen2014coarse})
\begin{equation}\label{eq:discreteHarmonic}
    \tilde{V}_h^{\partial}(\Omega_l)\coloneqq \{ v_h\in V_h(\Omega_l):~ a_{\epsilon,l}(v_h,w_h) = 0,\forall w_h\in V^0_{h}(\Omega_l) \}.
\end{equation}
Then we have $V_h^{\partial}(\Omega_l)=\tilde{V}_h^{\partial}(\Omega_l)$.
\begin{proof} It is clear that $V_h^{\partial}(\Omega_l)\subset\tilde{V}_h^{\partial}(\Omega_l)$. For any $v_h\in \tilde{V}_h^{\partial}(\Omega_l)$, define $\lambda_h\in
V_h(\partial\Omega_l)$ by
$$ \langle\lambda_h, \bar{w}_h\rangle_{\partial\Omega_l}=a_{\epsilon,l}(v_h,w_h), \quad\forall w_h\in V_{h}(\Omega_l). $$
Noting that $v_h\in \tilde{V}_h^{\partial}(\Omega_l)$, such a function $\lambda_h$ is well-defined. Then we have $v_h\in V_h^{\partial}(\Omega_l)$.

\end{proof}

Next we establish a key auxiliary result.
\begin{lemma}\label{lemma:new4.1} The global projection
    $P^{\rho}_{\epsilon} v_h$ can be reformulated as
    \begin{equation}\label{eq:globalProjReform}
        P^{\rho}_{\epsilon} v_h = (I-P^{\rho}_{0})\sum_{l=1}^{N}D_l E_l v_{h,l}^{\partial} + v_h,
    \end{equation}
    where $v_{h,l}^\partial$ is defined as $v_{h,l}^\partial \coloneqq P_{\epsilon,l} v_h  -R_l v_h$, and it just belongs to the local Helmholtz-harmonic space $V_h^\partial(\Omega_l)$.
\end{lemma}
\begin{proof}
    Recalling that $\{\chi_l\}_{l=1}^N$ is a partition of unity and using the definition of $D_l$, we can rewrite $v_h$ as
    \begin{equation*}
        v_h = \sum_{l=1}^{N}\pi_h (\chi_l v_h) =\sum_{l=1}^N D_l v_h = \sum_{l=1}^{N} D_l E_l R_l v_h,
    \end{equation*}
    where the last equality holds from the fact that $\chi_l$ is supported on $\overline{\Omega}_l$ and $E_lR_l v_h - v_h$ vanishes on $\overline{\Omega}_l$. From the above decomposition, we have
    \begin{equation*}
        P^{\rho}_{0} v_h = v_h - (I-P^{\rho}_{0})v_h = v_h - (I-P^{\rho}_{0})\sum_{l=1}^{N} D_l E_l R_l v_h.
    \end{equation*}
    Combining it with \eqref{eq:globalProj} leads to
   $$ P^{\rho}_{\epsilon} v_h=(I - P^{\rho}_{0})\sum_{l=1}^{N} D_l E_l P_{\epsilon,l}v_h+ P^{\rho}_{0}v_h=v_h+(I-P^{\rho}_{0})\sum_{l=1}^{N} D_l E_l(P_{\epsilon,l}v_h-R_l v_h), $$
  which implies the equality \eqref{eq:globalProjReform}.

    Every $w_{h,l}\in V^{0}_{h}(\Omega_l)$ vanishes on $\partial\Omega_l\setminus\partial\Omega$, so $E_l w_{h,l}$ vanishes on $\Omega\backslash\Omega_l$ and $E_l w_{h,l}\in V_h(\Omega)$.
    Then, by the definition of $P_{\epsilon,l}$, we deduce that
    \begin{align*}
        a_{\epsilon,l}(v_{h,l}^\partial,w_{h,l})
         & = a_{\epsilon,l}(P_{\epsilon,l} v_h  -R_l v_h, w_{h,l})                 \\
         & = a_{\epsilon}(v_h,E_l w_{h,l})-a_{\epsilon,l}(R_l v_h, w_{h,l}) \\
         & = a_{\epsilon,l}(v_h, w_{h,l})-a_{\epsilon,l}(v_h, w_{h,l}) = 0, \quad\forall w_{h,l}\in V^{0}_{h}(\Omega_l).
    \end{align*}
  This means that $v_{h,l}^\partial\in \tilde{V}_h^{\partial}(\Omega_l)$, which implies that $v_{h,l}^\partial \in V_h^\partial(\Omega_l)$ by {\bf Proposition 4.1}.
\end{proof}

For a function $v_h\in V_h(\Omega)$, define
\begin{equation}
v_h^\partial = \sum_{l=1}^{N}D_l E_l v_{h,l}^\partial\quad\mbox{with}\quad v_{h,l}^\partial \coloneqq P_{\epsilon,l} v_h  -R_l v_h.\label{new4.7}
\end{equation}
Since $v_{h,l}^\partial \in V_h^\partial(\Omega_l)$ and the generalized eigenvalue problem \eqref{eq:geEigenvalue}
is defined on $V_h^\partial(\Omega_l)$, we can expect that $v_h^\partial\in V_h^\partial(\Omega)$ could be well approximated by functions in $V^{\rho}_{h,0}(\Omega)$.
This will be verified later.

The following lemma is a direct result of Assumption \ref{assump:finiteOverlap}. It is shown that the energy of the summation of items can be estimated by the sum of energies of these items
(see \cite[Lemma 4.2]{graham2017domain}).
\begin{lemma}\label{lemma:sumEnergyBound} Let $v_l\in H^1(\Omega)$ satisfy $\supp{v_l}\subseteq \overline{\Omega}_l$. Under Assumption \ref{assump:finiteOverlap} (ii), we have
    \begin{equation}\label{eq:sumEnergyBound}
        \left\Vert \sum_{l=1}^{N} v_l\right\Vert_{1,\kappa}^2 \lesssim \sum_{l=1}^{N} \Vert v_l \Vert_{1,\kappa}^2
    \end{equation}
\end{lemma}
\begin{proof}
    Using Assumption \ref{assump:finiteOverlap}(ii) and Cauchy-Schwarz inequality, we can obtain
    \begin{align*}
        \left\Vert \sum_{l=1}^{N} v_l\right\Vert_{1,\kappa}^2
         & = \sum_{l=1}^{N}\left(v_l,\sum_{k=1}^{N}v_k\right)_{1,\kappa} = \sum_{l=1}^{N}\left(v_l,\sum_{k\in\Lambda(l)}v_k\right)_{1,\kappa}                                                           \\
         & \leq \left(\sum_{l=1}^{N}\Vert v_l\Vert_{1,\kappa}^2\right)^{\frac{1}{2}} \left(\sum_{l=1}^{N}\left(\sum_{k\in\Lambda(l)}\left\Vert  v_k\right\Vert_{1,\kappa}\right)^2\right)^{\frac{1}{2}} \\
         & \leq \left(\sum_{l=1}^{N}\Vert v_l\Vert_{1,\kappa}^2\right)^{\frac{1}{2}} \left(\sum_{l=1}^{N}\#\Lambda(l)\sum_{k\in\Lambda(l)}\left\Vert v_k\right\Vert_{1,\kappa}^2\right)^{\frac{1}{2}}   \\
         & \lesssim \sum_{l=1}^{N}\Vert v_l\Vert_{1,\kappa}^2.
    \end{align*}
\end{proof}

The following lemma indicates that the function $v_h^{\partial}$ defined by (\ref{new4.7}) can be approximated by a function in $V^{\rho}_{h,0}(\Omega)$:

\begin{lemma}\label{lemma:vhpApprox} Let Assumption \ref{assump:finiteOverlap} (ii) be satisfied.
    Set $z_h = \sum_{l=1}^{N}D_l E_l \Pi_{l}^\rho v_{h,l}^\partial$, then $z_h\in V^{\rho}_{h,0}(\Omega)$ and we have
    \begin{equation}
        \Vert v_h^\partial - z_h\Vert_{1,\kappa} \lesssim \rho \left(\sum_{l=1}^{N}\Vert v_{h,l}^\partial \Vert_{1,\kappa,\Omega_l}^2\right)^{\frac{1}{2}}.\label{new4.1}
    \end{equation}
\end{lemma}
\begin{proof}
    It is a direct result of Lemma \ref{lemma:eigenEstimate} and Lemma \ref{lemma:sumEnergyBound}. In fact, we have
    \begin{align*}
        \Vert v_h^\partial - z_h\Vert_{1,\kappa}^2
         & = \Vert \sum_{l=1}^{N}D_l E_l(I-\Pi_{l}^\rho)v_{h,l}^\partial\Vert_{1,\kappa}^2                  \\
         & \lesssim \sum_{l=1}^{N} \Vert D_l E_l(I-\Pi_{l}^\rho)v_{h,l}^\partial\Vert_{1,\kappa,\Omega_l}^2 \\
         & \leq \rho^2\sum_{l=1}^{N} \Vert v_{h,l}^\partial \Vert_{1,\kappa,\Omega_l}^2,
    \end{align*}
    where the first inequality follows from \eqref{eq:sumEnergyBound} and the last inequality follows from \eqref{eq:eigenEstimate}.
\end{proof}

In the following we estimate $\Vert v_{h,l}^\partial \Vert_{1,\kappa,\Omega_l}$.
For each subdomain $\Omega_l$, we define $\widetilde{\Omega}_l$ as
    \begin{equation}\label{eq:bdryStripe}
        \overline{\widetilde{\Omega}_l} = \mathop\bigcup_{\tau\in \widetilde{\mathcal T}_{h,l}} \overline{\tau},\quad \text{where }~~
        \widetilde{\mathcal T}_{h,l}\coloneqq\{\tau\in\mathcal{T}_h:~ \overline{\tau}\cap\overline{\Omega}_l\neq\emptyset\}.
    \end{equation}
Intuitively speaking, the subdomain $\widetilde{\Omega}_l$ is generated by enlarging $\Omega_l$ with one layer elements.

The following lemma gives a bound of $\Vert v_{h,l}^\partial \Vert_{1,\kappa,\Omega_l}$ controlled by $\Vert v_h\Vert_{1,\kappa,\widetilde{\Omega}_l}$.
\begin{lemma}\label{lemma:vhpBound} Let Assumption \ref{assump:meshsize} and Assumption \ref{assump:finiteOverlap} (i) be satisfied. Then,
 \begin{equation}\label{eq:vhpBound}
\Vert v_{h,l}^\partial \Vert_{1,\kappa,\Omega_l} \lesssim \min\left\lbrace {1+\kappa d_l,\frac{\kappa^2}{\epsilon}}\right\rbrace [1+(\kappa d_l)^{-1}]^{\frac{3}{2}}(\kappa h)^{-\frac{1}{2}} \Vert v_h\Vert_{1,\kappa,\widetilde{\Omega}_l}.
    \end{equation}


\end{lemma}
\begin{proof} By the definition of $v_{h,l}^\partial$, for every $w_{h,l}\in V_{h}(\Omega_l)$ we have
\begin{align}\label{eq:vhpImpData}
    a_{\epsilon,l}(v_{h,l}^\partial,w_{h,l})
     & = a_{\epsilon,l}(P_{\epsilon,l}v_h,w_{h,l}) - a_{\epsilon,l}(R_l v_h, w_{h,l}) \cr
     & = a_{\epsilon}(v_h, E_l w_{h,l}) - a_{\epsilon,l}(R_l v_h, w_{h,l})          \cr
     & \coloneqq F_l(w_{h,l}).
\end{align}
This shows that $v_{h,l}^\partial$ solves a local impedance problem with data $F_l$ (see \eqref{eq:localVF}). We first estimate a upper bound of $|F_l(w_{h,l})|$.

We have proved, in the second part of Lemma \ref{lemma:new4.1}, that
    \begin{align*}
        F_l(w_{h,l})
         & = a_{\epsilon}(v_h, E_l w_{h,l}) - a_{\epsilon,l}(R_l v_h, w_{h,l})      \\
         & = a_{\epsilon,l}(R_l v_h, w_{h,l}) - a_{\epsilon,l}(R_l v_h, w_{h,l}) = 0
    \end{align*}
   for any $w_{h,l}\in V^{0}_{h}(\Omega_l)$.
   Now we consider the general situation. Let $\mathcal{N}(\overline{\Omega}_l)$ be the set consisting of all finite element nodes on $\overline{\Omega}_l$ and ${\mathcal{N}}^{\partial}_l\subseteq \mathcal{N}(\overline{\Omega}_l)$
   consist of those located on $\partial\Omega_l\setminus\partial\Omega$. For any $x\in\mathcal{N}^{\partial}_l$, define $w^{\partial}_{h,l}\in V_{h}(\Omega_l)$ by its nodal values as
    \begin{equation*}
        w^{\partial}_{h,l}(x) =
        \left\lbrace
        \begin{aligned}
             & w_{h,l}(x),\quad &  & \text{if } x \in\mathcal{N}^{\partial}_l, \\
             & \quad 0,\quad          &  & \text{otherwise}.
        \end{aligned}
        \right.
    \end{equation*}
    Then $w_{h,l}-w^{\partial}_{h,l}\in V^0_{h}(\Omega_l)$. According to the previous discussion, we have
    \begin{align}
        F_l(w_{h,l})
         & = a_{\epsilon}(v_h, E_l w_{h,l}) - a_{\epsilon,l}(R_l v_h, w_{h,l})                          \notag             \\
         & = a_{\epsilon}(v_h, E_l w^{\partial}_{h,l}) - a_{\epsilon,l}(R_l v_h, w^{\partial}_{h,l}).\label{eq:reformFl}
    \end{align}
    For the estimate of the first term in \eqref{eq:reformFl}, we note that $E_l w^{\partial}_{h,l}$ vanishes on all elements except for those touching $\partial\Omega_l\setminus\partial\Omega$,
    which implies that $E_l w^{\partial}_{h,l} = 0$ on $\Omega\setminus\widetilde{\Omega}_l$. Therefore, from the continuity of $a_{\epsilon}(\cdot,\cdot)$ we have
        \begin{align}\label{eq:contBound}
            \vert a_{\epsilon}(v_h, E_l w^{\partial}_{h,l}) \vert
             & = \vert(\nabla v_h,\nabla E_l w^{\partial}_{h,l})_{\widetilde{\Omega}_l} - (\kappa^2+\mathrm{i}\epsilon)(v_h,E_l w^{\partial}_{h,l})_{\widetilde{\Omega}_l} \cr
             & - \mathrm{i}\kappa\langle v_h,E_l w^{\partial}_{h,l}\rangle_{\partial\widetilde{\Omega}_l\cap\partial\Omega}\vert                                              \cr
             & \lesssim [1+(\kappa d_l)^{-1}]\Vert v_h\Vert_{1,\kappa,\widetilde{\Omega}_l} \Vert E_l w^{\partial}_{h,l}\Vert_{1,\kappa,\widetilde{\Omega}_l},
        \end{align}

    Using the inverse estimate of the finite element functions and \eqref{eq:dL2NormSim}, and noting that $E_l w^{\partial}_{h,l}$ vanishes on
     all nodes except for those on $\partial\Omega_l\setminus\partial\Omega$, we can obtain
        \begin{align}\label{eq:whlDisNorm}
            \Vert E_l w^{\partial}_{h,l}\Vert_{1,\kappa,\widetilde{\Omega}_l}^2
             & = \Vert\nabla E_l w^{\partial}_{h,l}\Vert_{0,\widetilde{\Omega}_l}^2 + \kappa^2\Vert E_l w^{\partial}_{h,l}\Vert_{0,\widetilde{\Omega}_l}^2 \cr
             & \lesssim (h^{-2} +\kappa^2)\Vert E_l w^{\partial}_{h,l}\Vert_{0,\widetilde{\Omega}_l}^2                                                      \cr
            &\sim (h^{-2} +\kappa^2)h^2\sum_{x\in\mathcal{N}(\widetilde{\Omega}_l)}w^2_{h,l}(x)                                                                \cr
             & \sim (h^{-2} +\kappa^2)h^2\sum_{x\in\mathcal{N}^{\partial}_l}w^2_{h,l}(x)                                                                \cr
              &\lesssim ((\kappa h)^{-1}+{\kappa h})\kappa \Vert w_{h,l}\Vert_{\partial\Omega_l\setminus\partial\Omega}^2,
        \end{align}
    where the last inequality follows by applying \eqref{eq:dL2NormSim} on $\partial\Omega_l\setminus\partial\Omega$. From \eqref{eq:multiTrace} and the Cauchy-Schwarz inequality, we have
        \begin{align}\label{eq:whlMultiTrace}
            \kappa \Vert w_{h,l}\Vert_{\partial\Omega_l\setminus\partial\Omega}^2
             & \lesssim \frac{\kappa}{d_l}\Vert w_{h,l}\Vert_{0,\Omega_l}^2 + \kappa\Vert \nabla w_{h,l}\Vert_{\Omega_l}\Vert w_{h,l}\Vert_{0,\Omega_l} \cr
             & \lesssim (1+(\kappa d_l)^{-1})\kappa^2\Vert w_{h,l}\Vert_{0,\Omega_l}^2+ \Vert \nabla w_{h,l}\Vert_{0,\Omega_l}^2                        \cr
             & \leq (1+(\kappa d_l)^{-1})\Vert w_{h,l}\Vert_{1,\kappa,\Omega_l}^2.
        \end{align}
    Combining \eqref{eq:contBound}, \eqref{eq:whlDisNorm} and \eqref{eq:whlMultiTrace}, and noting that $\kappa h\lesssim 1$, we can derive that
    \begin{equation*}
        \vert a_{\epsilon}(v_h, E_l w^{\partial}_{h,l}) \vert \lesssim [1+(\kappa d_l)^{-1}]^{\frac{3}{2}}(\kappa h)^{-\frac{1}{2}}\Vert v_h\Vert_{1,\kappa,\widetilde{\Omega}_l} \Vert w_{h,l}\Vert_{1,\kappa,\Omega_l}.
    \end{equation*}
    The estimate of the second term in \eqref{eq:reformFl} could be established by the same argument except that the domain considered here is $\Omega_l$ instead of $\widetilde{\Omega}_l$.
    And we have
    \begin{equation*}
        \vert a_{\epsilon,l}(R_l v_h, w^{\partial}_{h,l}) \vert \lesssim [1+(\kappa d_l)^{-1}]^{\frac{3}{2}}(\kappa h)^{-\frac{1}{2}}\Vert v_h\Vert_{1,\kappa,\Omega_l} \Vert w_{h,l}\Vert_{1,\kappa,\Omega_l}.
    \end{equation*}
    Combining these two bounds, we obtain that
    \begin{equation*}
        \vert F_l(w_{h,l})\vert \lesssim [1+(\kappa d_l)^{-1}]^{\frac{3}{2}}(\kappa h)^{-\frac{1}{2}} \Vert v_h\Vert_{1,\kappa,\widetilde{\Omega}_l}\Vert w_{h,l}\Vert_{1,\kappa,\Omega_l}.
    \end{equation*}
    The estimate \eqref{eq:vhpBound} then follows from (\ref{eq:inf-sup}) (see Remark 2.1).
\end{proof}

\section{Convergence analysis for the case with a slightly large $\epsilon$}\label{sec:theoryBigep}
In this section, we consider the case without particular assumption on the shape of $\Omega$.

The next theorem gives a bound of $\Vert (I-P^{\rho}_{0})v_h^\partial \Vert_{1,\kappa}$, which will eventually lead to the final convergence result of this section.
\begin{theorem} Let $v_{h}^\partial$ be defined by (\ref{new4.7}). Under Assumption \ref{assump:finiteOverlap} (ii), we have
    \begin{equation}\label{eq:vhpProjError}
        \Vert (I-P^{\rho}_{0})v_h^\partial \Vert_{1,\kappa}
        \lesssim \frac{\rho\kappa^2}{\epsilon} \left(\sum_{l=1}^{N}\Vert v_{h,l}^\partial \Vert_{1,\kappa,\Omega_l}^2\right)^{\frac{1}{2}}.
    \end{equation}
\end{theorem}
\begin{proof}
    Recall that $z_h\in V^{\rho}_{h,0}(\Omega)$ defined in Lemma \ref{lemma:vhpApprox} satisfies $(I-P^{\rho}_{0})z_h = 0$. Then by the continuity and coercivity of $a_{\epsilon}(\cdot,\cdot)$, i.e. \eqref{eq:continuous} and \eqref{eq:coercive}, we have
    \begin{align*}
        \Vert (I-P^{\rho}_{0})v_h^\partial \Vert_{1,\kappa}^2
         & \lesssim \frac{\kappa^2}{\epsilon} \left\vert a_{\epsilon}((I-P^{\rho}_{0})v_h^\partial,(I-P^{\rho}_{0})v_h^\partial)\right\vert \\
         & = \frac{\kappa^2}{\epsilon} \left\vert a_{\epsilon}((I-P^{\rho}_{0})v_h^\partial,(I-P^{\rho}_{0})(v_h^\partial-z_h))\right\vert  \\
         & = \frac{\kappa^2}{\epsilon} \left\vert a_{\epsilon}((I-P^{\rho}_{0})v_h^\partial,v_h^\partial-z_h)\right\vert                      \\
         & \lesssim \frac{\kappa^2}{\epsilon} \Vert (I-P^{\rho}_{0})v_h^\partial \Vert_{1,\kappa} \Vert v_h^\partial-z_h \Vert_{1,\kappa},
    \end{align*}
    namely,
    \begin{equation*}
        \Vert (I-P^{\rho}_{0})v_h^\partial \Vert_{1,\kappa}
        \lesssim \frac{\kappa^2}{\epsilon} \Vert v_h^\partial-z_h \Vert_{1,\kappa}
        \lesssim \frac{\rho\kappa^2}{\epsilon} \left(\sum_{l=1}^{N}\Vert v_{h,l}^\partial \Vert_{1,\kappa,\Omega_l}^2\right)^{\frac{1}{2}},
    \end{equation*}
    where the last inequality holds from Lemma \ref{lemma:vhpApprox}.
\end{proof}

Before giving the final result, we make an assumption on the parameters.
\begin{assumption}\label{assump:epdh}
    The parameters $\epsilon$ and $d$ satisfy
    \begin{equation}
        \epsilon=\kappa^{1+\alpha}\quad \mbox{and}\quad d\sim\kappa^{-\beta},
    \end{equation}
    where $\alpha$ and $\beta$ satisfy $0\leq \alpha\leq\beta\leq 1$.
\end{assumption}

\begin{theorem}\label{thm:theoryBigep}
    Let Assumptions \ref{assump:meshsize}, \ref{assump:finiteOverlap} and \ref{assump:epdh} be satisfied. Define $\sigma = 2-(\alpha+\beta)+\frac{\gamma}{2}$.
    Then there exists a constant $C_0$ independent of the parameters $\kappa,\epsilon,h,d,\delta$ such that, if $\rho$ is chosen to satisfy $\rho\kappa^{\sigma}\leq C_0$,
    the operator $P^{\rho}_{\epsilon}=B^{-1}_{\epsilon}A_{\epsilon}$ possesses the spectrum properties
    \begin{equation}
  \Vert P^{\rho}_{\epsilon} v_h\Vert_{1,\kappa}\lesssim \Vert v_h\Vert_{1,\kappa},\quad \forall v_h\in V_h(\Omega)\label{spectrum:4.1}
  \end{equation}
 and
 \begin{equation}
 \mid(v_h,P^{\rho}_{\epsilon} v_h)_{1,\kappa}\mid\gtrsim \Vert v_h\Vert_{1,\kappa}^2,\quad \forall v_h\in V_h(\Omega)\label{spectrum:4.2}.
  \end{equation}
\end{theorem}
\begin{proof} Notice that, under Assumption \ref{assump:epdh}, we have $\kappa d_l\sim\kappa^{1-\beta}\gtrsim 1$ and $\kappa^2/\epsilon=\kappa^{1-\alpha}\gtrsim\kappa^{1-\beta}$,
which implies that
$$ \min\left\lbrace {1+\kappa d_l,\frac{\kappa^2}{\epsilon}}\right\rbrace\lesssim 1+\kappa d_l\lesssim \kappa^{1-\beta}. $$
Then, it follows from (\ref{eq:vhpBound}) and Assumption \ref{assump:meshsize} that
\begin{equation}\label{new4.8}
        \Vert v_{h,l}^\partial \Vert_{1,\kappa,\Omega_l} \lesssim \kappa^{1-\beta}\cdot \kappa^{\frac{\gamma}{2}} \Vert v_h\Vert_{1,\kappa,\widetilde{\Omega}_l}.
    \end{equation}

Let $v_h^\partial$ be defined by (\ref{new4.7}).
    Combining \eqref{eq:vhpProjError} and \eqref{new4.8}, we deduce that
    \begin{align*}
        \Vert (I-P^{\rho}_{0})v_h^\partial \Vert_{1,\kappa}
         & \lesssim \rho \kappa^{1-\alpha} \left(\sum_{l=1}^{N}\Vert v_{h,l}^\partial \Vert_{1,\kappa,\Omega_l}^2\right)^{\frac{1}{2}}                              \\
         & \lesssim \rho\kappa^{2-\alpha+\gamma/2-\beta} \left(\sum_{l=1}^{N}\Vert v_h \Vert_{1,\kappa,\widetilde{\Omega}_l}^2\right)^{\frac{1}{2}} \\
         & \lesssim \rho\kappa^{\sigma}\Vert v_h\Vert_{1,\kappa},
    \end{align*}
    where the last inequality follows from Assumption \ref{assump:finiteOverlap} (ii) and the definition of $\sigma$. Hence, there exists a constant $C_0$ independent of the parameters $\kappa,\epsilon,h,d,\delta$ such that, if $\rho$ satisfies $\rho\kappa^{\sigma}\leq C_0$, we have
    \begin{equation*}
        \Vert (I-P^{\rho}_{0})v_h^\partial \Vert_{1,\kappa} \leq \frac{1}{2} \Vert v_h\Vert_{1,\kappa}.
    \end{equation*}
    It follows from \eqref{eq:globalProjReform} that
    \begin{align*}
        \Vert P^{\rho}_{\epsilon}v_h \Vert_{1,\kappa}
         & = \Vert (I-P^{\rho}_{0})v_h^\partial + v_h\Vert_{1,\kappa}                          \\
         & \leq \Vert (I-P^{\rho}_{0})v_h^\partial\Vert_{1,\kappa} + \Vert v_h\Vert_{1,\kappa} \\
         & \leq\frac{3}{2}\Vert v_h\Vert_{1,\kappa},
    \end{align*}
    which gives (\ref{spectrum:4.1}). On the other hand, using \eqref{eq:globalProjReform} again, we obtain
    \begin{align*}
        \mid(v_h,P^{\rho}_{\epsilon} v_h)_{1,\kappa}\mid
         & = \mid(v_h,(I-P^{\rho}_{0})v_h^\partial + v_h)_{1,\kappa}\mid                                                           \\
         & = \mid\Vert v_h\Vert_{1,\kappa}^2 + (v_h,(I-P^{\rho}_{0})v_h^\partial)_{1,\kappa}\mid                                   \\
         & \geq \Vert v_h\Vert_{1,\kappa}^2 - \Vert v_h\Vert_{1,\kappa} \Vert (I-P^{\rho}_{0})v_h^\partial\Vert_{1,\kappa} \\
         & \geq \frac{1}{2} \Vert v_h\Vert_{1,\kappa}^2.
    \end{align*}
This leads to (\ref{spectrum:4.2}).
\end{proof}

\begin{remark}\label{remark:4.1}
In Theorem \ref{thm:theoryBigep}, the star-shaped assumption on $\Omega$ does not required, so $\epsilon$ should be not smaller than $\kappa$ (otherwise, $\rho$ must be very small).
This restriction on $\epsilon$ is mild, since the stiffness matrix of the Helmholtz equation (1.1) with $\epsilon=\kappa$ is a good preconditioner for the stiffness matrix of the
Helmholtz equations (1.1) with small parameters $\epsilon\geq 0$.
\end{remark}



\begin{remark} In this section we have not assumed that $d/\delta$ is bounded (i.e., with large overlap), but the dimension of the local coarse space $V^{\rho}_{h,0}(\Omega_l)$
with a small overlap size $\delta$ may be larger than that with a large overlap size $\delta$ for the same tolerance $\rho$. 
\end{remark}

\section{Convergence analysis for the case with a small $\epsilon$}\label{sec:theorySmallep}
In this section, we consider the case with a small absorption parameter $\epsilon$. To this end, we first give an assumption.
\begin{assumption}\label{assump:new5.1}
The domain $\Omega$ is star-shaped with respect to a disk. The parameters $\epsilon$, $d$ and $\delta$ satisfy
$$ 0\leq \epsilon\leq \kappa,\quad d/\delta \lesssim 1, \quad d\sim \kappa^{-(1+\tau)}~~\mbox{with}~~0\leq\tau\leq {1\over 3}.$$
\end{assumption}

For $v_h\in V_h(\Omega)$, set $e_h=(I-P^{\rho}_{0})v_h$. Let $w_h\in V_h(\Omega)$ be the solution to the following adjoint problem:
    \begin{equation}\label{eq:proofLitepBoundDual}
        a_{\epsilon}(\phi_h,w_h) = (\phi_h,e_h),\quad \forall \phi_h\in V_h(\Omega).
    \end{equation}
\begin{lemma}\label{lemma:new5.1} Let Assumptions \ref{assump:meshsize}, \ref{assump:finiteOverlap}, \ref{assump:new5.1} be satisfied. Then, for the function $w_h$ defined by (\ref{eq:proofLitepBoundDual}),
there exists a sufficiently small constant $C_0$ such that when $\kappa d\leq C_0$, the following estimate holds
\begin{equation}
\inf\limits_{w_h^\rho\in V_h^{\rho}(\Omega)}\left\Vert w_h-w_h^\rho\right\Vert_{1,\kappa}\lesssim (d+\rho)\Vert e_h\Vert_{0,\Omega}.\label{new5.1}
\end{equation}
\end{lemma}
\begin{proof} Let $\bar{a}_{\epsilon}(\cdot,\cdot)$ denote the conjugate form of $a_{\epsilon}(\cdot,\cdot)$. Notice that the domain $\Omega$ is star-shaped with respect to a disk.
It follows from Lemma \ref{lemma:infsup} (see Remark 2.1) that
    \begin{align}
        \Vert w_h\Vert_{1,\kappa}
         & \lesssim (1+\kappa)\sup_{\phi_h\in V_h(\Omega)}\frac{\vert \bar{a}_{\epsilon}(w_h,~\phi_h)\vert}{\Vert \phi_h\Vert_{1,\kappa}}
         =(1+\kappa)\sup_{\phi_h\in V_h(\Omega)}\frac{\vert a_{\epsilon}(\phi_h,~w_h)\vert}{\Vert \phi_h\Vert_{1,\kappa}}\cr
         & \lesssim (1+\kappa)\sup_{\phi_h\in V_h(\Omega)}\frac{\vert(\phi_h,\bar{e}_h)\vert}{\Vert \phi_h\Vert_{1,\kappa}}
         \lesssim (1+\kappa)  \sup_{\phi_h\in V_h(\Omega)}\frac{\Vert\phi_h\Vert_{0,\Omega} \Vert e_h\Vert_{0,\Omega}}{\Vert \phi_h\Vert_{1,\kappa}} \cr
         & \lesssim \frac{1+\kappa}{\kappa}\Vert e_h\Vert_{0,\Omega}\lesssim \Vert e_h\Vert_{0,\Omega}.\label{new5.2}
    \end{align}
    Let $w_{h,l} = R_lw_h\in V_h(\Omega_l)$. Define $w^{\partial}_{h,l}\in V_h^\partial(\Omega_l)$ as
    \begin{equation*}
        \left\lbrace
        \begin{aligned}
             & a_{\epsilon,l}(w^{\partial}_{h,l},\phi_{h,l}) = 0, &  & \forall \phi_{h,l}\in V_{0,h}(\Omega_l),           \\
             &\quad  w^{\partial}_{h,l} = w_{h,l},                      &  & \text{on }\partial\Omega_l\setminus\partial\Omega.
        \end{aligned}
        \right.
    \end{equation*}
    Moreover, let $w_h^\rho = \sum_{l=1}^{N}D_lE_l\Pi_l^{\rho} w^{\partial}_{h,l}\in V^{\rho}_{h,0}(\Omega)$ and $w^0_{h,l} = w_{h,l}-w^{\partial}_{h,l}\in V^0_{h}(\Omega_l)$. We can immediately get
    \begin{align}
        \left\Vert w_h-w_h^\rho\right\Vert_{1,\kappa}^2
         & =\left\Vert \sum_{l=1}^{N}D_lE_l w^0_{h,l}+\sum_{l=1}^{N}D_l E_l (w^{\partial}_{h,l}-\Pi_l^\rho w^{\partial}_{h,l})\right\Vert_{1,\kappa}^2\notag                                                                    \\
         & \lesssim \sum_{l=1}^{N}\left\Vert D_l E_l w^0_{h,l}\right\Vert_{1,\kappa,\Omega_l}^2 + \sum_{l=1}^{N}\left\Vert D_l E_l(I-\Pi_l^\rho) w^{\partial}_{h,l}\right\Vert_{1,\kappa,\Omega_l}^2
         \label{eq:proofLitepBound1}.
    \end{align}
    For the first summation of \eqref{eq:proofLitepBound1}, from interpolation error estimate refer to (see Lemma 3.3 of \cite{graham2020domain}), we have
    \begin{align}
        \sum_{l=1}^{N}\left\Vert D_l E_l w^0_{h,l}\right\Vert_{1,\kappa,\Omega_l}^2
         & \lesssim \sum_{l=1}^{N}\left\Vert \chi_l w^0_{h,l}\right\Vert_{1,\kappa,\Omega_l}^2                               \cr
         & +\left((1+\kappa h)\frac{h}{\delta}\right)^2\sum_{l=1}^{N}\left\Vert  w^0_{h,l}\right\Vert_{1,\kappa,\Omega_l}^2.\label{new5.10}
    \end{align}
    Note that $w^0_{h,l}$ vanishes on $\partial\Omega_l\setminus\partial\Omega$, then there holds the Poincar\'e inequality
    \begin{equation*}
        \Vert w^0_{h,l}\Vert_{0,\Omega_l}^2\lesssim d^2 \Vert \nabla(w^0_{h,l})\Vert_{0,\Omega_l}^2,
    \end{equation*}
    hence
    \begin{equation}\label{eq:proofLitepBound2}
        \Vert w^0_{h,l}\Vert_{1,\kappa,\Omega_l}^2
        \lesssim (1+(\kappa d)^2)\Vert \nabla(w^0_{h,l})\Vert_{0,\Omega_l}^2.
    \end{equation} Moreover, recall that $0\leq\chi_l\leq 1$ and $\vert\nabla\chi_l\vert\lesssim \delta^{-1}$, we can derive
    \begin{align}
         & \Vert\chi_l w^0_{h,l}\Vert_{0,\Omega_l}^2 \lesssim \Vert w^0_{h,l}\Vert_{0,\Omega_l}^2 \lesssim d^2 \Vert \nabla(w^0_{h,l})\Vert_{0,\Omega_l}^2,\label{eq:proofLitepBound3} \\
         & \begin{aligned}\label{eq:proofLitepBound4}
               \Vert\nabla(\chi_l w^0_{h,l})\Vert_{0,\Omega_l}^2
                & \lesssim \Vert w^0_{h,l}\nabla\chi_l \Vert_{0,\Omega_l}^2 + \Vert\chi_l\nabla w^0_{h,l}\Vert_{0,\Omega_l}^2   \\
                & \lesssim \frac{1}{\delta^2}\Vert w^0_{h,l} \Vert_{0,\Omega_l}^2 + \Vert \nabla w^0_{h,l} \Vert_{0,\Omega_l}^2 \\
                & \lesssim \left(1+\frac{d^2}{\delta^2}\right) \Vert \nabla w^0_{h,l} \Vert_{0,\Omega_l}^2.
           \end{aligned}
    \end{align}
    Substituting \eqref{eq:proofLitepBound2}, \eqref{eq:proofLitepBound3} and \eqref{eq:proofLitepBound4} into (\ref{new5.10}), yields
    \begin{equation}\label{eq:proofLitepBound5}
        \sum_{l=1}^{N}\left\Vert D_l E_l w^0_{h,l}\right\Vert_{1,\kappa,\Omega_l}^2 \lesssim \sum_{l=1}^{N}\left\Vert \nabla w^0_{h,l}\right\Vert_{0,\Omega_l}^2
    \end{equation}
    provided that $\kappa h\lesssim 1$, $\kappa d\lesssim 1$ and $\frac{d}{\delta}\lesssim 1$.
    For the second summation of \eqref{eq:proofLitepBound1}, from Lemma \ref{lemma:eigenEstimate}, there holds
    \begin{equation}\label{eq:proofLitepBound6}
        \begin{aligned}
             & \sum_{l=1}^{N}\left\Vert D_l E_l(I-\Pi_l^\rho) w^{\partial}_{h,l}\right\Vert_{1,\kappa,\Omega_l}^2
            \lesssim \rho^2\sum_{l=1}^{N}\left\Vert w^{\partial}_{h,l}\right\Vert_{1,k,\Omega_l}^2                                \\
             & \lesssim \rho^2\sum_{l=1}^{N}\left\Vert w^0_{h,l}\right\Vert_{1,k,\Omega_l}^2 + \rho^2\Vert w_{h}\Vert_{1,\kappa}^2 \\
             & \lesssim \rho^2\sum_{l=1}^{N}\Vert \nabla w^0_{h,l}\Vert_{0,\Omega_l}^2 + \rho^2\Vert e_h\Vert^2_{0,\Omega},
        \end{aligned}
    \end{equation}
    where the second inequality follows from triangular inequality and Assumption \ref{assump:finiteOverlap} (ii), and the last inequality follows from \eqref{eq:proofLitepBound2}
    (noting that $\kappa d\lesssim 1$)
    and (\ref{new5.2}). Inserting \eqref{eq:proofLitepBound5} and \eqref{eq:proofLitepBound6} into \eqref{eq:proofLitepBound1}, we can finally get
    \begin{equation}\label{eq:proofLitepBound7}
        \left\Vert w_h-w_h^\rho\right\Vert_{1,\kappa}^2 \lesssim (1+\rho^2) \sum_{l=1}^{N}\Vert \nabla w^0_{h,l}\Vert_{0,\Omega_l}^2 + \rho^2\Vert e_h\Vert^2_{0,\Omega}.
    \end{equation}
    In the following, we estimate $\Vert \nabla w^0_{h,l}\Vert_{0,\Omega_l}$. Since $w^0_{h,l}\in V_{0,h}(\Omega_l)\subset V_h(\Omega)$, we can derive
    \begin{align*}
        a_{\epsilon,l}(w^0_{h,l},w^0_{h,l})
         & = a_{\epsilon,l}(w_{h,l},w^0_{h,l}) - a_{\epsilon,l}(w^{\partial}_{h,l},w^0_{h,l})               \\
         & = \overline{a_{\epsilon,l}(w^0_{h,l},w_{h,l})} -2\mathrm{i}\epsilon(w_{h,l},\bar{w}^0_{h,l})_{\Omega_l} \\
         & = (e_h-2\mathrm{i}\epsilon w_{h,l},\bar{w}^0_{h,l})_{\Omega_l}.
    \end{align*}
    On the other hand, there exists a constant $c_0$ such that when $c_0 \kappa d\leq \frac{1}{\sqrt{2}}$, there holds
    \begin{align*}
        \mathop{\mathrm{Re}}a_{\epsilon,l}(w^0_{h,l},w^0_{h,l})
         & = \Vert \nabla w^0_{h,l}\Vert_{0,\Omega_l}^2 - \kappa^2
        \Vert  w^0_{h,l}\Vert_{0,\Omega_l}^2                                                                          \\
         & \geq  \Vert \nabla w^0_{h,l}\Vert_{0,\Omega_l}^2 - (c_0\kappa d)^2 \Vert  \nabla w^0_{h,l}\Vert_{0,\Omega_l}^2 \\
         & \geq  \frac{1}{2}\Vert \nabla w^0_{h,l}\Vert_{0,\Omega_l}^2.
    \end{align*}
    Therefore, we have
    \begin{equation*}
        \Vert \nabla w^0_{h,l}\Vert_{\Omega_l} \lesssim d(\Vert e_h\Vert^2_{0,\Omega_l}+\epsilon\Vert w_{h,l}\Vert_{0,\Omega_l}).
    \end{equation*}
    Combining it with \eqref{eq:proofLitepBound7} and noting that $\epsilon\lesssim \kappa$, we get
    \begin{align*}
        \left\Vert w_h-w_h^\rho\right\Vert_{1,\kappa}^2
         & \lesssim d^2 \sum_{l=1}^{N}(\Vert e_h\Vert^2_{0,\Omega_l}+\epsilon^2 \Vert w_{h,l}\Vert_{0,\Omega_l}^2) + \rho^2\Vert e_h\Vert^2_{0,\Omega} \\
         & \lesssim d^2 \Vert e_h\Vert^2_{0,\Omega} + d^2 \Vert w_{h}\Vert_{1,\kappa}^2+ \rho^2\Vert e_h\Vert^2_{0,\Omega}                                    \\
         & \lesssim (d^2+\rho^2)\Vert e_h\Vert^2_{0,\Omega}.
    \end{align*}
Here we have used (\ref{new5.2}) again. This leads to the desired inequality.
\end{proof}

\begin{theorem} Under the assumptions made in Lemma \ref{lemma:new5.1},
there exists a sufficiently small constant $C_0$ such that when $\kappa d\leq C_0$ and $\kappa \rho\leq C_0$, we have
\begin{equation}
    \Vert (I-P^{\rho}_{0})v_h\Vert_{1,\kappa}\lesssim \Vert v_h\Vert_{1,\kappa},\quad\forall v_h\in V_h(\Omega) \label{new5.3}
\end{equation}
and
\begin{equation}
    \Vert (I-P^{\rho}_{0})v^{\partial}_h\Vert_{1,\kappa}\lesssim
\rho \left(\sum_{l=1}^{N}\Vert v_{h,l}^\partial \Vert_{1,\kappa,\Omega_l}^2\right)^{\frac{1}{2}}
,\quad\forall v^{\partial}_h=\sum_{l=1}^{N}D_lE_l v^{\partial}_{h,l}\in V^{\partial}_h(\Omega).\label{new5.4}
\end{equation}
\end{theorem}
\begin{proof}
    From the definition of $a_\epsilon(\cdot,\cdot)$ and recalling that $e_h\coloneqq (I-P^{\rho}_{0})v_h$ satisfies
    \begin{equation}\label{eq:proofLitepBoundOrth}
        a_{\epsilon}(e_h,\phi_h) = 0,\quad\forall \phi_h\in V^{\rho}_{h,0}(\Omega),
    \end{equation}
    we can obtain
    \begin{equation}\label{eq:proofLitepBoundToL2}
        \begin{aligned}
            \Vert e_h\Vert_{1,\kappa}^2
             & = \mathop{\mathrm{Re}}a_\epsilon(e_h,e_h) +2\kappa^2\Vert e_h\Vert^2_{0,\Omega}                       \\
             & = \mathop{\mathrm{Re}}a_\epsilon(e_h,v_h)+2\kappa^2\Vert e_h\Vert^2_{0,\Omega}                        \\
             & \lesssim \Vert e_h\Vert_{1,\kappa} \Vert v_h\Vert_{1,\kappa} + 2\kappa^2\Vert e_h\Vert_{0,\Omega}^2.
        \end{aligned}
    \end{equation}
    We will estimate $\Vert e_h\Vert_{0,\Omega}$ by duality argument. Let $w_h$ be the solution to the following adjoint problem (\ref{eq:proofLitepBoundDual}).
    Taking $\phi_h = e_h$ in \eqref{eq:proofLitepBoundDual} and using \eqref{eq:proofLitepBoundOrth}, there holds
    \begin{align*}
        \Vert e_h\Vert^2_{0,\Omega}
         & = a_\epsilon(e_h,w_h)                                                    \\
         & = a_\epsilon(e_h,w_h-w_h^{\rho})                                         \\
         & \lesssim \Vert e_h\Vert_{1,\kappa} \Vert w_h -w_h^{\rho}\Vert_{1,\kappa} \\
         & \lesssim (d+\rho)\Vert e_h\Vert_{1,\kappa}\Vert e_h\Vert_{0,\Omega},
    \end{align*}
    namely,
    \begin{equation}
        \Vert e_h\Vert_{0,\Omega} \lesssim (d+\rho)\Vert e_h\Vert_{1,\kappa}. \label{new5.5}
    \end{equation}
    Combining it with \eqref{eq:proofLitepBoundToL2}, we can derived that
    \begin{equation*}
        \Vert e_h\Vert_{1,\kappa}\lesssim \Vert v_h\Vert_{1,\kappa},
    \end{equation*}
    provided that $\kappa d\lesssim 1$ and $\kappa \rho\lesssim 1$, where the implicit constant is small enough. This gives (\ref{new5.3}).

 When $e_h=(I-P^{\rho}_{0})v^{\partial}_h$, define $v^{\partial,\rho}_h=\sum_{l=1}^{N}D_lE_l\Pi_l^\rho v^{\partial}_{h,l}\in V^{\rho}_{h,0}(\Omega)$.
 Then, by Lemma \ref{lemma:vhpApprox} we get
   \begin{equation*}
        \begin{aligned}
            \Vert e_h\Vert_{1,\kappa}^2
             & = \mathop{\mathrm{Re}}a_\epsilon(e_h,e_h) +2\kappa^2\Vert e_h\Vert^2_{0,\Omega}                       \\
             & = \mathop{\mathrm{Re}}a_\epsilon(e_h,v^{\partial}_h-v^{\partial,\rho}_h)+2\kappa^2\Vert e_h\Vert^2_{0,\Omega}                        \\
             & \lesssim \rho\Vert e_h\Vert_{1,\kappa}\left(\sum_{l=1}^{N}\Vert v_{h,l}^\partial \Vert_{1,\kappa,\Omega_l}^2\right)^{\frac{1}{2}}+ 2\kappa^2\Vert e_h\Vert_{0,\Omega}^2.
        \end{aligned}
    \end{equation*}
 Substituting (\ref{new5.5}) into the above inequality and using the assumptions, yields (\ref{new5.4}).
 \end{proof}

When the parameters satisfy Assumption \ref{assump:new5.1}, the estimate (\ref{eq:vhpBound}) becomes
$$  \Vert v_{h,l}^\partial \Vert_{1,\kappa,\Omega_l} \lesssim (\kappa d_l)^{-{3\over 2}}(\kappa h)^{-{1\over 2}}\Vert v_h\Vert_{1,\kappa,\widetilde{\Omega}_l}=\kappa^{3\tau+\gamma\over 2}\Vert v_h\Vert_{1,\kappa,\widetilde{\Omega}_l} \lesssim\kappa \Vert v_h\Vert_{1,\kappa,\widetilde{\Omega}_l}. $$
Substituting it into (\ref{new5.4}) and using Lemma \ref{lemma:new4.1}, we can prove the following main result
(in the same manner as the proof of Theorem \ref{thm:theoryBigep}):

\begin{theorem}\label{thm:theorySmallep}
    Let Assumptions \ref{assump:meshsize}, \ref{assump:finiteOverlap}, \ref{assump:new5.1} be satisfied.
    Then there exists a constant $C_0$ independent of the parameters $\kappa,\epsilon,~h,~d,\delta$ such that, if $\rho$ is chosen to satisfy $\rho\kappa\leq C_0$,
    the operator $P^{\rho}_{\epsilon}=B^{-1}_{\epsilon}A_{\epsilon}$ possesses the spectrum properties
    \begin{equation}
  \Vert P^{\rho}_{\epsilon} v_h\Vert_{1,\kappa}\lesssim \Vert v_h\Vert_{1,\kappa},\quad \forall v_h\in V_h(\Omega)\label{spectrum:5.1}
  \end{equation}
 and
 \begin{equation}
 \mid(v_h,P^{\rho}_{\epsilon} v_h)_{1,\kappa}\mid\gtrsim \Vert v_h\Vert_{1,\kappa}^2,\quad \forall v_h\in V_h(\Omega)\label{spectrum:5.2}.
  \end{equation}
 \end{theorem}

\begin{remark} In Theorem \ref{thm:theorySmallep}, we have not assumed $\epsilon\geq\kappa$, but we need to assume that the domain $\Omega$ is star-shaped with respect to a ball and the subdomain size $d$ should satisfy that $d\kappa$ is sufficiently small. As we will see in numerical experiments, the condition that $d\kappa$ is sufficiently small is unnecessary. Besides, the condition $d/\delta \lesssim 1$ is not absolutely necessary, namely, the case with a small overlap $\delta$ is also allowed. For this case, a larger coarse space is required in theory.
\end{remark}

\begin{remark} Theorem \ref{thm:theoryBigep} and Theorem \ref{thm:theorySmallep} indicate that the preconditioned GMRES iteration with the preconditioner $B^{-1}_{\epsilon}$ for solving the discrete system (\ref{new2.1}) possesses uniformly convergence rate independent of $h$ and $\kappa$ under suitable assumptions.
\end{remark}

\section{Some other coarse spaces}

In this section, we compare the newly proposed coarse space with some other related coarse spaces, and introduce a cheaper coarse space.

\subsection{Several existing coarse spaces}\label{subsec:otherCS}
There are several coarse spaces working for Helmholtz equations: the coarse grid space \cite{cai1992domain, graham2017domain}; the DtN coarse space \cite{conen2014coarse} and HGenEO coarse space \cite{bootland2021comparison}.


Throughout this subsection, for simplicity and without causing confusion, we will use the same notations used in Subsection 3.2 when introducing different coarse spaces and corresponding preconditioners.
\subsubsection{Coarse space based on the coarse grid}
The non-overlapping domain decomposition $\{\Omega'_l\}_{l=1}^N$ may form a coarse grid (triangulation) ${\mathcal T}_d$ on $\Omega$, then the simplest coarse space is defined on the grid ${\mathcal T}_d$.
This kind of coarse space was first analyzed in \cite{cai1992domain} for indefinite elliptic equations satisfying some assumptions. This coarse space was extended to solving the Helmholtz equations
with large absorption parameters in \cite{graham2017domain}.


Consider a more general coarse grid $\mathcal{T}_H$ with the mesh size $H\geq h$ built on $\Omega$. Let $V_H(\Omega)$ be the corresponding finite element space. We call $V_H(\Omega)$ the grid coarse space. Note that in general, $V_H(\Omega)$ is not necessarily a subspace of $V_h(\Omega)$. In this case, we could define $E_0$ as the interpolation operator $\pi_h$ restricted on $V_H(\Omega)$ and the operator $A_{\epsilon,0}$ is defined as usual. These give the corresponding preconditioner. Later in the numerical experiments, we will focus on the case with $H\sim d$.

\subsubsection{The DtN coarse space}
The Dirichlet-to-Neumann (DtN) coarse space, which is first proposed and analyzed in \cite{nataf2011coarse,dolean2012analysis} for Laplace equations, is based on local eigenvalue problems on subdomain interfaces related to the DtN map. We will introduce the DtN coarse space first proposed in \cite{conen2014coarse} for Helmholtz equations.

The original DtN coarse space focus on the pure Helmholtz equation (with $\epsilon=0$). Here, for consistency, we still consider the model with a positive absorption parameter. For $\Omega_l$, let $\Gamma_l$ be the inner boundary, i.e. $\Gamma_l\coloneqq \partial\Omega_l\setminus\partial\Omega$. Given Dirichlet data $u_{\Gamma_l}$ on $\Gamma_l$, we can define the Helmholtz harmonic extension $u\in H^1(\Omega_l)$ as the solution of the following problem:

\begin{equation*}
    \left\lbrace
    \begin{aligned}
         & -\Delta u -(\kappa^2+\mathrm{i}\epsilon)u = 0,  &  & \text{in }\Omega_l,                          \\
         & \partial_{\mathbf{n}}u -\mathrm{i}\kappa u = 0, &  & \text{on }\partial\Omega_l\setminus\Gamma_l, \\
         & u = u_{\Gamma_l}                                &  & \text{on }\Gamma_l.
    \end{aligned}
    \right.
\end{equation*}
Then the DtN map of Dirichlet data $u_{\Gamma_l}$ is the Neumann data of $u$, namely
\begin{equation*}
    \textrm{DtN}_{\Omega_l}(u_{\Gamma_l}) = \partial_{\mathbf{n}}u\vert_{\Gamma_l}.
\end{equation*}
And the eigenvalue problem to be solved is given by
\begin{equation*}
    \textrm{DtN}_{\Omega_l}(u_{\Gamma_l}) = \lambda u_{\Gamma_l}.
\end{equation*}
Using the Green's formula and recall that $\tilde{V}_h^\partial(\Omega)$ denote the discrete Helmholtz-harmonic space with Dirichlet boundary data (see (\ref{eq:discreteHarmonic})), the discretized weak form of above eigenvalue problem writes: find $\xi\in V_h^\partial(\Omega)$ such that
\begin{equation}\label{eq:dtnEigen}
    \widetilde{a}_{\epsilon,l}(\xi,\theta) = \lambda \left\langle \xi,\bar{\theta}\right\rangle_{\Gamma_l},\quad \forall \theta\in \tilde{V}_h^\partial(\Omega_l),
\end{equation}
where
\begin{equation}\label{eq:neuSesq}
    \widetilde{a}_{\epsilon,l}(\xi,\theta)\coloneqq(\nabla\xi,\nabla\theta )_{\Omega_l}-(\kappa^2+\mathrm{i}\epsilon)(\xi,\theta)_{\Omega_l}-\mathrm{i}\kappa\langle\xi,\theta\rangle_{\partial\Omega_l\setminus\Gamma_l}
\end{equation}
is the sesquilinear form of the local Neumann problem. Similar to Subsection \ref{subsec:coarseSpace}, letting $\{(\lambda_l^i,\xi_l^i)\}_{i=1}^{m_l}$ be the eigenpairs of \eqref{eq:dtnEigen}, the local coarse space can be defined as
\begin{equation*}
    V_{h,0}^\rho(\Omega_l)\coloneqq\textrm{span}\left\lbrace \xi_l^i:~ \mathrm{Re}(\lambda_l^i)\leq \rho^2\right\rbrace,
\end{equation*}
where $\rho> 0$ is a given tolerance. The DtN coarse space is then built following the same procedure in Subsection \ref{subsec:coarseSpace}.

\subsubsection{The HGenEO coarse space}
The GenEO (Generalized Eigenproblems in the Overlap) coarse space was first defined and analyzed in \cite{spillane2014abstract} for abstract symmetric positive definite problems. In \cite{bootland2021comparison} a GenEO-type coarse space was proposed for Helmholtz equations, which is called HGenEO coarse space. We will briefly described it as follows.

Similarly, we focus on the Helmholtz problem with an positive absorption parameter since the pure problem is merely a specialization when $\epsilon=0$. In \cite{bootland2021comparison} the authors linked the
underlying Helmholtz problem to the positive definite Laplace problem, designing the following local eigenproblem posed on $V_h(\Omega_l)$: find $\xi\in V_h(\Omega_l)$ such that
\begin{equation}\label{eq:hgeneoEigen}
    \widetilde{a}_{\epsilon,l}(\xi,\theta) = \lambda(\nabla(D_l E_l \xi),\nabla(D_l E_l \theta))_{\Omega_l},\quad \forall \theta\in V_h(\Omega_l),
\end{equation}
where $\widetilde{a}_{\epsilon,l}(\cdot,\cdot)$ is defined by \eqref{eq:neuSesq} and $D_l,E_l$ is defined as in Subsection \ref{subsec:1level}. As usual, let $\{(\lambda_l^i,\xi_l^i)\}_{i=1}^{\textrm{dim}V_h(\Omega_l)}$ be the eigenpairs of \eqref{eq:hgeneoEigen}. Then the local coarse space is defined as
\begin{equation*}
    V_{h,0}^\rho(\Omega_l)\coloneqq\textrm{span}\left\lbrace \xi_l^i:~ \mathrm{Re}(\lambda_l^i)\leq \rho^2\right\rbrace,
\end{equation*}
where $\rho> 0$ is a given tolerance. Again, the HGenEO coarse space can be built as in Subsection \ref{subsec:coarseSpace}.
\subsubsection{Comparisons of the spectral coarse spaces}

It can be seen from \eqref{eq:dtnEigen} and \eqref{eq:hgeneoEigen} that, for DtN coarse space, eigenproblems posed on discrete Helmholtz-harmonic space $\tilde{V}_h^\partial(\Omega_l)$ are solved while for HGenEO coarse space, eigenproblems to be solved are posed on $V_h(\Omega_l)$. This leads to the fact that the size of eigenproblems defining the DtN coarse space is much smaller than the size of eigenproblems defining HGenEO coarse space. For both coarse spaces, indefinite eigenproblems are solved and the real part of the eigenvalues are used as a threshold criterion. This might include some redundant information and make the corresponding coarse space larger.
In our new approach, we choose different sesquilinear forms in the generalized eigenvalue problem \eqref{eq:geEigenvalue}. We change the sesquilinear form in both sides of \eqref{eq:hgeneoEigen} into $(\cdot,\cdot)_{1,\kappa,\Omega_l}$ and pose the eigenproblems on $V_h^\partial(\Omega_l)$, which defines a positive semi-definite eigenproblems with relatively smaller size. We believe that the proposed coarse space can capture more precise information compared to the DtN coarse space and the HGenEO coarse space. This will be shown later in the numerical experiments.

\begin{remark} By {\bf Proposition 4.1}, the spaces $V_h^{\partial}(\Omega_l)$ and $\tilde{V}_h^{\partial}(\Omega_l)$ are the same in theory. But from implemental point of view, there are differences between them.
The basis functions of the space $V_h^{\partial}(\Omega_l)$ can be directly obtained by applying $A^{-1}_{\epsilon,l}$ appearing in (\ref{new4.0}).
However, the basis functions of the space $\tilde{V}_h^{\partial}(\Omega_l)$ have to be re-computed by solving Dirichlet boundary problems, which brings extra costs. In fact, these
Dirichlet boundary problems may be not well-defined unless $\epsilon>0$ or $\kappa d\ll 1$.
\end{remark}

\subsection{An economic coarse space}\label{subsec:ecoCS}
In the construction of the spectral coarse spaces, the main cost of computation comes from solving the considered eigenproblems. In this subsection, we introduce an economic coarse space in order to avoid solving eigenproblems.

We still start from the construction of the local coarse spaces as in Subsection \ref{subsec:coarseSpace}. In Subsection \ref{subsec:coarseSpace}, local coarse space on $\Omega_l$ can be viewed as a good approximation of the discrete Helmholtz-harmonic space $V_h^\partial(\Omega_l)$ (see Lemma \ref{lemma:eigenEstimate}). If the local coarse space could be built directly, we can expect that it will preform well when it is a good approximation of $V_h^\partial(\Omega_l)$.

A discrete Helmholtz-harmonic function $v_h\in V_h^\partial(\Omega_l)$ is fully determined by a function $\lambda_h\in V_h(\Gamma_l)$:
\begin{equation*}
   a_{\epsilon,l}(v_h,w_h)=\langle \lambda_h,w_h\rangle_{\Gamma_l}, \quad\forall w_h\in V_h(\Omega_l).
\end{equation*}
Intuitively, $\lambda_h$ can be viewed as the ``discrete impedance data" of $v_h$ on $\Gamma_l$. Therefore, a good approximation of $\lambda_h\in V_h(\Gamma_l)$ will lead to a good approximation of $v_h$. In order to achieve this, the most natural idea is to consider the finite element space based on a ``coarse" grid on $\Gamma_l$. However, this idea is impractical since the shape of $\Gamma_l$ is not a straight line. Fortunately, in 2D there exists a continuous map from $[0,1]$ to $\Gamma_l$ and a linear map $\Phi_l:L^2([0,1])\rightarrow L^2(\Gamma_l)$ can be induced. Then, given $\nu>0$ and the corresponding mesh size $h_{\nu}=\frac{1}{\nu}$, the finite element space $V_{h_{\nu}}([0,1])$ could be built and $V_h^{\nu}(\Gamma_l)\coloneqq \pi_h \Phi_l (V_{h_{\nu}}([0,1]))$ is an approximation of $V_h(\Gamma_l)$. Finally, the economic local coarse is defined as
\begin{equation*}
    V_{h,0}^{\nu}(\Omega_l)\coloneqq \left\lbrace {\mathcal E}_{\epsilon,l}(\lambda_h):~ \lambda_h\in V_h^{\nu}(\Gamma_l)\right\rbrace,
\end{equation*}
where ${\mathcal E}_{\epsilon,l}(\lambda_h)$ for $\lambda_h\in V_h^{\nu}(\Gamma_l)$ is defined such that
\begin{equation*}
    a_{\epsilon,l}({\mathcal E}_{\epsilon,l}(\lambda_h), w_h) = \langle \lambda_h,\bar{w}_h\rangle_{\Gamma_l},\quad\forall w_h\in V_h(\Omega_l).
\end{equation*}
As in subsection \ref{subsec:coarseSpace}, the economic coarse space is defined by
\begin{equation*}
    V^{\nu}_{h,0}(\Omega):=\{v_h=\sum_{l=1}^N D_l E_l v_l:~v_l\in V_{h,0}^{\nu}(\Omega_l)\},
\end{equation*}
and the corresponding preconditioner could be built in the same way.

\section{Numerical experiments}
In this section, we display some numerical results to illustrate the efficiency of the proposed preconditioners.

As usual, the domain $\Omega$ we considered in the numerical experiments is the unit square $[0,1]^2$ in $\mathbb{R}^2$. Throughout this section, we make use of standard $p$-order($p=1$ or $2$ specifically) finite element approximation and the structured triangle grid of size $h \approx \kappa^{-\frac{2p+1}{2p}}$. Under these settings, the pollution error of the finite element solution of Helmholtz equation can be essentially reduced
according to the error estimates established in \cite{DuWu2015}. At first let the domain be decomposed into several non-overlapping square (or rectangle) subdomains. Then a few layers of fine mesh elements
are added to make up the overlapping subdomains. The following two strategies will be used in this section:
\begin{enumerate}
    \item {\it Minimal overlap.} Only one layer is added, and in this case $\delta = h$.
    \item {\it Generous overlap.} Multiple layers are added, ensuring that $\delta\approx \frac{d}{4}$.
\end{enumerate}
If not specifically mentioned, the latter strategy, i.e. generous overlap is utilized.

We consider the same example as in \cite{graham2020domain}: the data $f,g$ in \eqref{eq:shiftHelmholtz} is chosen such that $f = 0$ and $g$ is the impedance data of a plane wave $u(\textbf{x}) = \exp(\mathrm{i}\kappa \textbf{x}\cdot \widehat{d})$ where $\widehat{d} = (1/\sqrt{2},1/\sqrt{2})$.
Standard GMRES with zero initial guess is used. The stopping criterion is that the relative $L^2$-norm of the residual of the iterative solution is less than $10^{-6}$ and the maximum number of iterations is set as 1000. The $\times$ in the tables means that GMRES doesn't converge when the maximum number of iterations is reached.

The implementation is mainly based on MFEM\cite{anderson2021mfem}, which is a free, lightweight, scalable C\texttt{++} library for finite element methods. Moreover, we use MUMPS\cite{amestoy2001fully}, which is
integrated in PETSc\cite{balay2020petsc}, both as the subdomain solver and the coarse solver. When it comes to the construction of the spectral coarse spaces, multiple generalized eigenvalue problems are solved, and SLEPc\cite{hernandez2005slepc} is used to do these computations. The numerical experiments in this section are done on the high performance computers of State Key Laboratory of Scientific and Engineering Computing, CAS.

\subsection{Efficiency of the proposed preconditioner}\label{subsec:efficiency}
In Sections \ref{sec:theoryBigep} and \ref{sec:theorySmallep}, different requirements on the tolerance $\rho$ ensuring convergence is proposed for the case with a relatively large $\epsilon$ or a small $\epsilon$ (see Theorem \ref{thm:theoryBigep} and \ref{thm:theorySmallep}). The first experiment is designed to verify the theories. In this experiment we focus on the meaningful case $\epsilon=\kappa$ ($\alpha=0$) with $\rho \sim \kappa^{\beta-2- \frac{\gamma}{2}}$ and $\epsilon = 0$ with $\rho\sim \kappa^{-1}$, the former is to verify Theorem \ref{thm:theoryBigep} while the latter is to verify Theorem \ref{thm:theorySmallep}. As mentioned earlier, we choose $h\approx\kappa^{-\frac{2p+1}{2p}}$ where $p=1,2$ is the order of nodal finite elements, implying that $\gamma = \frac{1}{2p}$ in Assumption \ref{assump:meshsize}.

We report the iteration counts of the GMRES method with the proposed preconditioner and the ratio of size of coarse space to maximal size of local problems. The numerical results for both linear finite element approximation
and quadratic finite element approximation are showed in Table \ref*{tab:theroyCheck}.

\begin{table}[H]\small
    \caption{\rm GMRES iteration counts and the ratio of size of coarse space to maximal size of local problems (in parentheses)}
    \label{tab:theroyCheck}%
     \begin{center}
        \begin{tabular}{|c||c|c|c||c|c|c|}
            \hline
            \multicolumn{7}{|c|}{linear FE approximation} \\
            \hline
            & \multicolumn{3}{c||}{$\epsilon=\kappa$} & \multicolumn{3}{c|}{$\epsilon=0$}\\
            \hline
            $\kappa\backslash\beta$  & 0.7 & 0.6 & 0.5 & 1.0 & 0.8 & 0.6\\
            \hline
            $30\pi$  & 5(1.4) & 4(0.5) & 4(0.12) & 7(208.7)& 9(5.0) & 13(0.23)\\
            \hline
            $40\pi$  & 4(2.8) & 3(0.6) & 3(0.14) & 6(233.2)& 6(8.0) & 10(0.24)\\
            \hline
            $50\pi$  & 3(3.5) & 3(0.9)  & 2(0.21) & 6(493.9)& 6(8.2) & 7(0.28)\\
            \hline
            $60\pi$    & 3(4.5) & 2(1.0) & 2(0.23) & 4(875.4)& 5(11.8) & 6(0.29)\\
            \hline
            \hline
            \multicolumn{7}{|c|}{quadratic FE approximation} \\
            \hline
            & \multicolumn{3}{c||}{$\epsilon=\kappa$} & \multicolumn{3}{c|}{$\epsilon=0$}\\
            \hline
            $\kappa\backslash\beta$  & 0.7 & 0.6 & 0.5 & 1.0 & 0.8 & 0.6\\
            \hline
            $40\pi$  & 6(3.7) & 4(1.1) & 3(0.3) & 6(507.8)& 6(22.4) & 10(0.7)\\
            \hline
            $60\pi$  & 3(11.9) & 2(2.1) & 3(0.6) & 6(1965.9)& 6(32.0) & 6(0.9)\\
            \hline
            $80\pi$  & 3(18.7) & 2(3.7)  & 2(0.7) & 4(5064.8)& 4(73.4) & 5(1.4)\\
            \hline
            $100\pi$  & 2(25.9) & 2(4.4) & 2(1.0) & 4(7714.3)& 3(113.6) & 4(1.5)\\
            \hline
        \end{tabular}%
    \end{center}
\end{table}

 The data in Table \ref{tab:theroyCheck} indicate that the iteration counts are almost independent of the wave numbers, which verify Theorems \ref{thm:theoryBigep} and \ref{thm:theorySmallep}
 under different settings. In fact, for the theoretical choice of $\rho$ (and fixed $\beta$), the iteration counts decrease slightly but the size of the coarse space increases rapidly
 when increasing $\kappa$.
 This phenomenon means that the requirements on $\rho$ made in Theorem \ref{thm:theoryBigep} and Theorem \ref*{thm:theorySmallep} are too strict, which leads to very small $\rho$ and quite large coarse spaces.

 It can also be seen from Table \ref*{tab:theroyCheck} that the preconditioner show similar efficiency for both finite element approximations. From now on we will utilize the quadratic finite element approximation, which exhibits good performance when solving high-frequency Helmholtz problems.

According to the above discussions, the tolerance $\rho$ should be selected carefully in practice, taking into account the stability of iteration counts and the size of the corresponding coarse space.
Intuitively, as $\kappa$ and $d$ increase, every local problem becomes more indefinite, so the value of $\rho$ needs to be decreased slightly such that each local coarse space contains more basis functions.
Guided by this idea and some observations, we find that $\rho = \frac{1}{2}\kappa^{\frac{\beta-1}{2}}$ may be a good choice. We design another experiment with such choice of $\rho$ and list the data in
Table \ref*{tab:goodRho}.

\begin{table}[H]\small
    \caption{\rm GMRES iteration counts and the ratio of size of coarse space to maximal size of local problems (in parentheses) with $\rho = \frac{1}{2} \kappa^{\frac{\beta-1}{2}}$}
    \label{tab:goodRho}
       \begin{center}
        \begin{tabular}{|c||c|c|c||c|c|c|}
            \hline
            & \multicolumn{3}{c||}{$\epsilon=\kappa$} & \multicolumn{3}{c|}{$\epsilon=0$}\\
            \hline
            $\kappa\backslash\beta$ & 0.7 & 0.6 & 0.5 & 0.7 & 0.6 & 0.5\\
            \hline
            $40\pi$  & 7(3.3)  & 6(0.8)  & 6(0.15) & 7(3.3)  & 6(0.8)  & 6(0.15) \\
            \hline
            $80\pi$  & 7(6.5)  & 7(1.1)  & 6(0.18) & 7(6.5)  & 7(1.1)  & 7(0.18) \\
            \hline
            $120\pi$ & 8(6.4)  & 7(1.2)  & 7(0.17) & 9(6.4)  & 8(1.2)  & 8(0.17) \\
            \hline
            $160\pi$ & 10(8.1)  & 8(1.3)  & 6(0.17) & 12(8.1)  & 9(1.3)  & 7(0.17)\\
            \hline
        \end{tabular}
    \end{center}

\end{table}

It can be seen from Table \ref*{tab:goodRho} that, when increasing  $\kappa$, there is at most a slight increase in the iteration counts for a fixed $\beta$.
Comparing the results for the two choices of $\epsilon$, we find that the iteration counts barely change and the size of coarse spaces are almost the same.
This phenomenon can be explained by the fact that the Helmholtz problem (1.1) with $\epsilon\leq 1$ is close to the pure Helmholtz problem in the spectrum sense.
Meanwhile, as $\beta$ decreases, the ratio of size of the coarse space to the maximal size of local problems also decreases.
The results in Table \ref*{tab:goodRho} show that such choice of $\rho$ can balance the robustness and the size of coarse space.
We can also observe that the size of coarse space is comparable with that of local problems when $\beta=0.6$, which is an ideal choice of $\beta$ (for the quadratic FE approximations).

In Subsection \ref*{subsec:ecoCS}, we introduced a strategy to generate a coarse space without solving eigenproblems, which is named the economic coarse space. The final experiment in this subsection is devoted to show the efficiency of this economic coarse space. Recall that this coarse space is fully determined by the ``coarse'' mesh size $h_{\nu}$ and ``coarse'' finite element space $V_{\nu}([0,1])$ (see Subsection \ref*{subsec:ecoCS}). We choose $V_{\nu}([0,1])$ as the 1-d quadratic finite element space. Inspired by the choice of $\rho$ in the previous discussion, we choose $h_{\nu}=\kappa^{\beta-1}$. In Table \ref*{tab:eco}, we list the results for this case.

\begin{table}[H]\small
    \caption{\rm GMRES iteration counts and the ratio of size of coarse space to maximal size of local problems (in parentheses) for the economic coarse space with $h_{\nu}=\kappa^{\beta-1}$}
    \label{tab:eco}
 \begin{center}
        \begin{tabular}{|c||c|c|c||c|c|c|}
            \hline
            & \multicolumn{3}{c||}{$\epsilon=\kappa$} & \multicolumn{3}{c|}{$\epsilon=0$}\\
            \hline
            $\kappa\backslash\beta$ & 0.7 & 0.6 & 0.5 & 0.7 & 0.6 & 0.5\\
            \hline
            $40\pi$  & 4(4.3)  & 5(0.9)  & 4(0.21) & 5(4.3)  & 5(0.9)  & 5(0.21)  \\
            \hline
            $80\pi$  & 5(7.2)  & 4(1.4)  & 5(0.23) & 5(7.2)  & 5(1.4)  & 5(0.23) \\
            \hline
            $120\pi$ & 8(7.2)  & 5(1.4)  & 4(0.21) & 9(7.2)  & 5(1.4)  & 5(0.21) \\
            \hline
            $160\pi$ & 7(9.8)  & 6(1.5)  & 5(0.21) & 7(9.8)  & 6(1.5)  & 5(0.21) \\
            \hline
        \end{tabular}
    \end{center}
\end{table}

Table \ref*{tab:eco} indicates that the economic coarse space exhibits good efficiency with such choice of $h_{\nu}$, especially when $\beta$ is relatively small. We furthermore find from Table \ref*{tab:eco}
and Table \ref*{tab:goodRho} that the economic coarse space shows similar performance with the original coarse space. Again, considering the robustness of the iteration and the size of coarse space, we believe that
$\beta=0.6$ is a good choice, which leads to a relatively large size $d$ of subdomains.

\subsection{Comparisons with several related preconditioners}
In this subsection, we will compare the newly proposed economic preconditioner with existing preconditioners introduced in Subsection \ref*{subsec:otherCS}. We will focus on the case of $\epsilon=0,1$ and consider the representative case with $\beta=0.6,0.8$.

The first experiment is designed to compare our preconditioner with ORAS preconditioner and the two level preconditioner equipped with grid coarse space.
For convenience, we use ``ECS" to denote the economic coarse space and ``GCS'' to denote the grid coarse space. For ECS, we still set parameter $h_{\nu}=\kappa^{\beta-1}$.
Besides, for GCS we display the results for an additional case $\beta = 1$ when GCS performs well. For the rationality of comparison,
the solving time of linear systems is also counted. The numerical results are showed in Table \ref*{tab:comp1}.

\begin{table}[H]\small
    \caption{\rm GMRES iteration counts, the ratio of size of coarse space to maximal size of local problems (in parentheses, for two level methods) and solving time in seconds (in square brackets) for three preconditioners} for the case with $\epsilon=1$
    \label{tab:comp1}
    \begin{center}
        \begin{tabular}{|c||c|c||c|c||c|c|c|}
            \hline
            & \multicolumn{2}{c||}{ECS} & \multicolumn{2}{c||}{ORAS} & \multicolumn{3}{c|}{GCS} \\
            \hline
            $\kappa\backslash\beta$ & 0.8 & 0.6 & 0.8 & 0.6 &1.0 & 0.8 & 0.6\\
            \hline
            $40\pi$ & \makecell{{ 8(19.8)}\\ $[0.36]$} & \makecell{{ 5(0.9)}\\$[0.25]$} & \makecell{{ 280}\\$[1.70]$} & \makecell{{ 98}\\$[1.66]$} & \makecell{{ 9(405.2)}\\$[1.70]$} & \makecell{{ 32(13.3)}\\$[1.01]$} & \makecell{{ 238(0.26)}\\$[6.58]$} \\
            \hline
            $80\pi$ & \makecell{{ 5(42.8)}\\$[1.29]$}    & \makecell{{ 5(1.4)}\\$[0.96]$}   & \makecell{{ 477}\\ $[27.44]$}  & \makecell{{ 171}\\$[7.15]$} & \makecell{{ 10(1614.5)} \\ $[4.17]$} & \makecell{{ 161(21.5)} \\ $[18.30]$} & \makecell{{ 411(0.28)} \\ $[32.24]$}  \\
            \hline
            $120\pi$ & \makecell{{ 8(49.7)}\\$[3.50]$}  & \makecell{{ 5(1.4)}\\$[1.89]$} & \makecell{{ 653}\\ $[114.93]$}  & \makecell{{ 191}\\$[33.84]$} & \makecell{{ 11(2700.3)} \\ $[11.52]$} & \makecell{{ 722(25.0)} \\ $[227.46]$} & \makecell{{ 552(0.27)} \\ $[131.32]$} \\
            \hline
            $160\pi$ & \makecell{{ 11(66.7)}\\$[6.68]$}  & \makecell{{ 6(1.5)}\\$[3.95]$} & \makecell{{ 848}\\ $[269.68]$}  & \makecell{{ 232}\\$[83.12]$}  & \makecell{{ 12(4797.4)} \\ $[12.74]$} & $\times$ & \makecell{{ 234(0.25)} \\ $[95.06]$} \\
            \hline
        \end{tabular}
    \end{center}
\end{table}

 From Table \ref*{tab:comp1}, we can see that ORAS and GCS perform worse as expected for the cases of $\beta = 0.6, 0.8$, and GCS performs well when $\beta=1$ but the size of the corresponding coarse space is quite huge and almost unacceptable. Thanks to fast convergence of the proposed preconditioner, the solving time for ECS is obviously less than that of ORAS and GCS.

At the end of this subsection, we compare the proposed economic coarse space with DtN coarse space and HGenEO coarse space. In order to make a fair comparison, we use a fixed number $m$ of local basis functions per subdomain.
For spectral coarse space, this is easy to achieve by fixing $m$ ``bad" eigenvectors (related to the eigenvalues having smaller real parts) per subdomain while for the economic coarse space, the parameter $\nu$ (see Subsection \ref*{subsec:ecoCS}) is set to ${m\over 2}$ since the quadratic approximation is used.
Table \ref{tab:pureFixm} shows the comparison of the preconditioners for generous overlap and minimal overlap cases. For case of minimal overlap, the parameter $m$ is set slightly bigger to
reduce the impact of decreasing the overlap.

\begin{table}[H]\small
    \caption{\rm GMRES iteration counts for three preconditioners with fixed number of eigenvectors per subdomain for the case with $\epsilon=0$}
    \label{tab:pureFixm}
    \begin{tabular}{|c||c|c|c|c|c|c|c|c|}
        \hline
        \multicolumn{9}{|c|}{$\beta=0.8$}\\
        \hline
        \multirow{2}{*}{$\kappa$} & \multicolumn{4}{c|}{Minimal overlap} & \multicolumn{4}{c|}{Generous overlap}\\
        \cline{2-9}
        & $m$ &  ECS  & DtN & HGenEO & $m$ & ECS  & DtN   & HGenEO \\
        \hline
        $40\pi$ & 8 & 6 & 7 & 30 & 6 & 8 & 9 & 162 \\
        \hline
        $80\pi$ & 10 & 6 & 8 & 33 & 8& 5 & 639 & 312 \\
        \hline
        $120\pi$ & 10 & 8 & 12  & 61 & 8& 8 & 841 & 859 \\
        \hline
        $160\pi$ & 10 & 10 & 18  & 93 & 8& 11 & $\times$ & $\times$ \\
        \hline
        \hline
        \multicolumn{9}{|c|}{$\beta=0.6$}\\
        \hline
        \multirow{2}{*}{$\kappa$} & \multicolumn{4}{c|}{Minimal overlap} & \multicolumn{4}{c|}{Generous overlap}\\
        \cline{2-9}
        & $m$ &  ECS  & DtN & HGenEO & $m$ & ECS  & DtN   & HGenEO \\
        \hline
        $40\pi$ & 16 & 9 & 39 & 17 & 14 & 5 & 96 & 111 \\
        \hline
        $80\pi$ & 22 & 12 & 37 & 19 & 20 & 5 & 158 & 274 \\
        \hline
        $120\pi$ & 24 & 8 & 20 & 57 & 22 & 5 & 67 & 971  \\
        \hline
        $160\pi$ & 26 & 9 & 71 & 417 & 24 & 6 & 174 & $\times$  \\
        \hline
    \end{tabular}
\end{table}

From Table \ref*{tab:pureFixm}, ECS performs better than the other preconditioners, especially for the case of generous overlap or $\beta=0.6$. The case with minimal overlap is preferred in practice since less additional
cost is required for this case. The newly proposed preconditioner is also suitable for this case.

\bibliographystyle{siamplain}

\end{document}